\documentclass[11pt]{article}

\usepackage{amsmath,amsfonts,amsthm,amscd,amssymb,graphicx,mathtools}
\usepackage{cite}
\usepackage{bbm}
\usepackage{color}
\usepackage[english=american]{csquotes}
\usepackage[colorlinks=true, pdfstartview=FitV, linkcolor=black, citecolor=black, urlcolor=black]{hyperref}
\usepackage{calc}
\usepackage{mathptmx}
\usepackage{t1enc}
\usepackage{stackrel}
\usepackage{dsfont}
\usepackage{enumitem}
\usepackage{hyperref}
\usepackage{wasysym}
\usepackage[T1]{fontenc}

\newtheorem{theo}{Theorem}[section]
\newtheorem{lem}[theo]{Lemma}
\newtheorem{prop}[theo]{Proposition}

\newtheorem{cor}[theo]{Corollary}

\newtheorem{rem}[theo]{Remark}
\newtheorem{defi}[theo]{Definition}

\DeclareMathOperator{\esssup}{ess\,sup}

\DeclareMathOperator{\supp}{supp}

\newcommand{\T}{\mathbb{T}}

\newcommand{\N}{\mathbb{N}}
\newcommand{\R}{\mathbb{R}}

\newcommand{\eps}{\epsilon}

\newcommand{\dx}{ \mathrm{d}x}
\newcommand{\dt}{\mathrm{d}t}

\newcommand{\dd}{\mathrm{d}}

\newcommand{\image}{\mathrm{im}\;}

\renewcommand{\eps}{\varepsilon}
\renewcommand{\epsilon}{\varepsilon}

\def\Xint#1{\mathchoice
	{\XXint\displaystyle\textstyle{#1}}%
	{\XXint\textstyle\scriptstyle{#1}}%
	{\XXint\scriptstyle\scriptscriptstyle{#1}}%
	{\XXint\scriptscriptstyle\scriptscriptstyle{#1}}%
	\!\int}
\def\XXint#1#2#3{{\setbox0=\hbox{$#1{#2#3}{\int}$ }
		\vcenter{\hbox{$#2#3$ }}\kern-.58\wd0}}

\def\dashint{\Xint-}

\numberwithin{equation}{section}
\DeclareMathAlphabet{\mathcal}{OMS}{cmsy}{m}{n}

\begin{document}
\allowdisplaybreaks
	\title{Measure-valued low Mach number limits of ideal fluids}
	
	\author{Dennis Gallenm\"uller\footnotemark[1]}
	
	\date{}
	
	\maketitle
	
	\begin{abstract}
		As a framework for handling low Mach number limits we consider a notion of measure-valued solution of the incompressible Euler system which explicitly formulates the usually suppressed dependence on the pressure variable. We clarify in which sense such a measure-valued solution is generated as a low Mach limit and state sufficient conditions for this convergence. For the special case of pressure-free solutions we are able to give such sufficient conditions only on the incompressible solution itself. As a necessary condition for low Mach limits we obtain a Jensen-type inequality. The same necessary condition actually holds true for limits of weak solutions or vanishing viscosity limits. We illustrate that this Jensen inequality is not trivially fulfilled. Since measure-valued solutions in the classical sense are always generated by weak solutions, this shows that our solution concept contains more information and low Mach limits lead to a partial selection criterion for incompressible solutions.
	\end{abstract}
	
	\renewcommand{\thefootnote}{\fnsymbol{footnote}}
	
	\footnotetext[1]{Institute of Applied Analysis, Ulm University, Helmholtzstra\ss e 18, 89081 Ulm, Germany.\\Email: dennis.gallenmueller@uni-ulm.de}

\section{Introduction}
	All fluids in nature are viscous and compressible to at least a small extent. On the other hand, to simplify the physical models one tries to study the systems in the limit of zero viscosity and where the fluid is assumed to be incompressible. The hope is now to infer properties of the imperfect fluid from the results of the ideal case. Therefore, it is necessary to understand the limit process for the parameters governing the viscosity and the compressibility. In this article, we will mainly investigate the low Mach number limit, which describes the transition from compressible to incompressible fluids. Moreover, also some results concerning the vanishing viscosity will be discussed.\\
	We start by considering the isentropic compressible Euler equations in the nondimensionalized formulation, cf. \cite{KM81},
	\begin{equation}
		\begin{aligned}
			\partial_t(\rho^{\eps}u^{\eps})+\operatorname{div}(\rho^{\eps}u^{\eps}\otimes u^{\eps})+\nabla\frac{1}{\eps}(\rho^{\eps})^{\gamma}&=0,\\
			\partial_t\rho^{\eps}+\operatorname{div}(\rho^{\eps}u^{\eps})&=0,\label{eq:isentropiceulereps}
		\end{aligned}
	\end{equation}
	where the unknowns are the velocity $u$ and the density $\rho$ with upper index $\eps$. Here, $\sqrt{\eps}$ denotes the Mach number of the fluid, which measures the compressibility of the fluid and is proportional to the reciprocal of the speed of sound. The constitutive relation for the pressure function $P(\rho)=\frac{1}{\eps}\rho^{\gamma}$ with $\gamma>1$ models the isentropic motion of perfect gases. This choice for the pressure function is quite standard and will be, in fact, crucial for some of our results especially in Section~\ref{sect:thepressurefreecase} below for the case $\gamma=1+\frac{2}{d}$. However, most of the methods of the present paper can also be adapted to other well-behaved choices for the pressure function. Note that we will consider the above system only with periodic boundaries, so it has to be complemented just by suitable initial conditions. As a first criterion for physical solutions we may require the energy to be constant or at least non-increasing. This will be formalized in Section~\ref{sect:preparation} by using the total energy inequality. Although the existence of solutions fulfilling a stronger local energy inequality is known, we will stick to the weaker total energy admissibility as we will apply the results from \cite{GW21}, where only the total energy can be controlled.\\
	We chose to use the velocity formulation instead of the momentum formulation for the isentropic Euler system, since both are equivalent for densities bounded away from zero and the velocity formulation is formally closer to the incompressible limit system. This system is achieved by sending the Mach number $\sqrt{\eps}$ to zero. An informal argument, using an expansion of the dependent variables in terms of $\eps$, shows that a sequence of solutions of (\ref{eq:isentropiceulereps}) should converge to a solution of the incompressible Euler equations
	\begin{equation}
		\begin{aligned}
			\partial_tu+\operatorname{div}(u\otimes u)+\nabla P&=0,\\
			\operatorname{div}u&=0.\label{eq:incompressibleeuler}
		\end{aligned}
	\end{equation}
	The velocity $u$ is again unknown, but now also the pressure $P$ is a priori an unknown variable. This gives already a hint on the quite different roles the pressure plays in the compressible and incompressible situation. However, note that for weak solutions the pressure is in fact completely determined by the velocity via
	\begin{align*}
		P=(-\Delta)^{-1}\operatorname{div}\operatorname{div}(u\otimes u)
	\end{align*}
	due to the incompressibility condition. By weak solutions we mean solutions in the sense of distributions. Those weak solutions enable us to describe the behavior of less regular fluid flows like shock waves or turbulence as opposed to smooth strong solutions modeling laminar flow. An even weaker concept of solution was introduced by DiPerna to describe oscillation effects in limits of solutions, cf. \cite{D85}. We will call those \textit{classical measure-valued solutions}. Note that we do not have to deal with possible concentration effects, since we will only consider uniformly bounded sequences. Measure-valued solutions consist of parametrized probability measures, so-called Young measures, which solve the equations only in an averaged sense. Thus, these solutions give us only a probability distribution of the velocity at each point in spacetime. Weak solutions correspond then to the case of an atomic Young measure. The concept of measure-valued solutions provides us not only with a simpler existence theory, but also many limit processes, like vanishing viscosity sequences or convergence of numerical schemes, are so far only known to generate measure-valued solutions, see e.g. \cite{FL18}.\\
	The main purpose for the present study was to better understand the interplay of the compressible and incompressible models, and to examine the different roles of the pressure therein. Our first task is to come up with a framework to formulate the low Mach limit on the level of measure-valued solutions in order to find conditions on how to decide whether a solution is generated by a low Mach limit of compressible weak solutions or not. Our approach now is to generalize the notion of classical measure-valued solutions, where the pressure does not appear as variable due to the usage of divergence-free test functions. This is achieved through treating the pressure as independent variable by allowing all test functions, see Definition~\ref{defi:mvsdefinitions}. We will call those solutions \textit{augmented measure-valued solutions}. This provides us with a tool to discuss the transition from compressible to incompressible fluid dynamics on the level of measure-valued solutions. Hereby, the role of the pressure has to be quite prominent in contrast to the previously mentioned solution concepts, where the pressure is usually projected away. For the compressible system measure-valued solutions  have been developed in \cite{N93}. In this case, we do not need to generalize our solution concept, since we already test against not necessarily divergence-free functions and the compressible pressure appears explicitly. However, note that the pressure in the compressible case behaves quite differently as it is a prescribed function of the density variable.\\
	The limit process of low Mach numbers can be made rigorous for sufficiently regular strong solutions, cf. \cite{KM81} and \cite{E77}. 
	If the incompressible system admits a sufficiently regular strong solution to a given initial data, then also measure-valued low Mach sequences with appropriately convergent initial data converge to the incompressible classical solution, cf. \cite{FKM19}. The latter result is obtained by using the relative energy inequality, see \cite{GSW15}. Another consequence of the relative energy method is the weak-strong uniqueness property of classical incompressible measure-valued and weak solutions, as observed in \cite{BDS11}. In Section~\ref{sect:preparation} we will see that augmented measure-valued solutions possess this property at least in the velocity variable. However, we have to trade for the weakness of augmented solutions with the non-uniqueness in the pressure variable, which is only fixed up to the first moment.\\
	The problem now is that it is still an open question if in three space dimensions for every regular initial data the corresponding solution of the incompressible Euler equations stays regular globally. For non-smooth initial data the situation becomes in fact very rough. The method of convex integration introduced for fluid dynamics in \cite{DS09} and \cite{DS10} by C.~De~Lellis and L.~Sz\'ekelyhidi shows that there exists initial data for the incompressible Euler system admitting infinitely many energy admissible weak solutions. This happens actually for an $L^2$-dense set of initial data, which has been shown in \cite{SW12}. Therefore, in general we cannot expect that we are in the situation above of having a regular limit solution. On the other hand, the existence of weak solutions, hence in particular of (classical) measure-valued solutions, is guaranteed for every initial data, cf. \cite{W11}. However, note that the weak solutions obtained in \cite{W11} cannot satisfy any admissibility criterion as their total energy increases at the initial time. We give sufficient conditions for the existence of a low Mach limit in this general case, i.e.~without assuming the existence of a regular strong solution, in Section~\ref{sect:limitsoflowmachsequences}. Here, we also specify how the low Mach limit on the level of measure-valued solutions should be understood: An augmented solution $\mu$ is a low Mach limit if a certain lift $(P^{\eps}_{\bar{\rho}})_{\#}\nu^{\eps}$ of the low Mach sequence $\nu^{\eps}$ converges to $\mu$ in the sense of Young measures, cf. Theorem~\ref{theo:measuresequenceimpliessolution}, where $P_{\bar{\rho}}^{\eps}$ is the lift map introduced in Section~\ref{sect:preparation}.\\
	Note that the method of convex integration has been applied also to the compressible system with prescribed density, see e.g. \cite{C14}, \cite{DS10}, and \cite{F14}. But only recently a fully compressible generalization has been worked out in \cite{DSW21}, see also \cite{M21}.\\
	Other singular limits like the limit of weak solutions or the vanishing viscosity limit have been studied more extensively in the past years. Although it is still unknown if vanishing viscosity limits of Leray-Hopf solutions converge to weak solutions of the incompressible Euler system, it has been observed already in \cite{DM87} that such limits generate at least a classical measure-valued solution if the Leray-Hopf solutions are uniformly bounded. It is also an open question if every incompressible measure-valued solution is such a limit. However, for the compressible system it has been shown in \cite{CFKW15} and \cite{GW20} that there are di-atomic measure-valued solutions not coming from a vanishing viscosity sequence of compressible Navier-Stokes solutions. In fact, the latter solutions are also not generated by limits of compressible weak solutions. The key observation in these works is that the underlying states are not wave-cone-connected with respect to the linearly relaxed Euler system (\ref{eq:relaxedeulersystem}). This is deeply connected to the violation of a Jensen-type inequality. Such Jensen conditions first appeared in the calculus of variations, where D.~Kinderlehrer and P.~Pedregal used them to characterize gradient Young measures, cf. \cite{KP91} and \cite{KP94}. This idea has been extended by I.~Fonseca and S.~Müller in \cite{FM99} to general $\mathcal{A}$-free Young measures generated by $\mathcal{A}$-free sequences, where $\mathcal{A}$ is a linear homogeneous differential operator. In fact, such Jensen conditions in terms of $\mathcal{A}$-quasiconvex functions are the general criterion to decide if a compressible measure-valued solution is the limit of compressible weak solutions or not, cf. \cite{GW21}. Hereby, the relaxation to a linear homogeneous differential equation, a truncation method from \cite{G21} and \cite{M99}, and the ideas from \cite{FM99} are crucial.\\
	Contrary to that, in the incompressible situation every classical measure-valued solution is generated by a limit of weak solutions as has been shown in \cite{SW12}.
	On the other hand, there are augmented incompressible solutions which are not generated by weak limits or a vanishing viscosity limit, cf. Section~\ref{sect:examplesviolatingthejensencondition}. This seems to be a contradiction. The resolution of that lies in the fact that the Jensen condition corresponding to the relaxed system without density is trivially fulfilled if the freedom in the choice of the pressure is sufficiently large, which is the case for classical measure-valued solutions. For augmented measure-valued solutions we will not have this freedom anymore, since the pressure is already explicitly a part of the solution. In particular, the Jensen-type condition may not be trivially satisfied in this case. In fact, this Jensen condition is also a necessary condition for augmented solutions to be a low Mach limit. This will be our main result.
	\begin{theo}\label{theo:jensencond}
		Suppose that an augmented measure-valued solution $\mu$ is generated either by a low Mach number limit with uniformly bounded pressure lifts, or by a limit of uniformly bounded incompressible weak solutions with uniformly bounded pressure. Then $\mu$ satisfies the following Jensen condition
		\begin{align}
			\left\langle S_{\#}\mu,f\right\rangle\geq Q_{\mathcal{A}_E}f\left(\left\langle S_{\#}\mu,\operatorname{id}\right\rangle\right)\label{eq:jensencondition}
		\end{align}
		for all $f\in C(\R^N)$ on a subset of $(0,T)\times \T^d$ of full measure.
	\end{theo}
	Note that we do not need any conditions on the initial data for this result. The definitions of the differential operator $\mathcal{A}_E$ and of the lifts, in particular the function $S$, will be provided in Section~\ref{sect:preparation}. Similarly to the proofs in \cite{GW21} and \cite{CFKW15}, proving the above necessary condition highly relies on the pullback to the relaxed system (\ref{eq:relaxedeulersystem}). The crucial observation here is that the relaxed Euler system actually does not perceive the Mach number since the whole pressure term $\frac{1}{\eps}(\rho^{\eps})^{\gamma}$ is replaced by a variable. Therefore, the $\mathcal{A}_E$-free structure survives in the limit and we can use abstract considerations involving manipulation of quasiconvex envelopes and Jensen conditions for measures generated by $\mathcal{A}_E$-free sequences. Let us briefly mention that Theorem~\ref{theo:jensencond} also holds true in the case of other constitutive relations for the compressible pressure, since the explicit form of the pressure never appears in the proof due to the substitution in the relaxed Euler system.\\
	The above discussion indicates that some augmented solutions are in a sense unphysical. In view of the vast non-uniqueness, which is even given for smooth initial data at least in the pressure variable, we therefore suggest a novel partial selection criterion for incompressible solutions:\\
	\textit{Augmented measure-valued solutions of the incompressible Euler system should be discarded as unphysical if they are not generated by a low Mach limit of compressible weak solutions.}\\
	See also Remark \ref{rem:selectioncriterion} below. Note that we do not claim to obtain uniqueness from this selection. We only give a criterion to sort out some solutions which definitely fail to be a candidate for the physical solution. One may complement this by the selection due to vanishing viscosity limits. However, in a sense, the low Mach limit is better understood since we can give sufficient conditions on augmented solutions to be a low Mach limit at least in the pressure-free case, cf. Section~\ref{sect:thepressurefreecase}. For vanishing viscosity limits the situation is much more difficult as such sufficient conditions are not yet available. In this context, it should be also further investigated if low Mach limits actually generate classical measure-valued solutions, i.e. if the a priori only augmented solution is atomic in the pressure. However, this appears to be a challenging problem as the pressure then necessarily needs to be non-locally connected to the velocity by the inverse Laplacian.\\
	Let us summarize the scope of this paper. We will begin with stating the main definitions and giving some preliminary results in Section~\ref{sect:preparation}. In Section~\ref{sect:limitsoflowmachsequences} we provide some sufficient conditions for low Mach sequences to converge to an augmented solution. Conversely, in Section~\ref{sect:necessaryjensentypeconditionsforincompressiblesingularlimits} we obtain necessary conditions for low Mach limits in terms of a Jensen-type inequality, whereby we prove Theorem~\ref{theo:jensencond}. Subsequently, the pressure-free case is studied in Section~\ref{sect:thepressurefreecase} in which we obtain sufficient conditions on the augmented solution itself that guarantee the existence of a low Mach sequence generating this solution. The last section is then dedicated to finding di-atomic measures that violate the Jensen condition (\ref{eq:jensencondition}). In particular, those measures cannot be generated by low Mach sequences.
	
\section{Preparation}\label{sect:preparation}
Let us first introduce some notation which will be used throughout the article. Note that most of these definitions are non-standard terminology.\\
For a measurable function $f\colon X\to Y$ between measure spaces and a measure $\mu$ on $X$ denote by $f_\#\mu$ the pushforward measure on $Y$.\\
We define the following lifting maps:
\begin{itemize}
	\item $S\colon (u,P)\mapsto \left(1,u,u\ocircle u,P+\frac{|u|^2}{d}\right)$,
	\item $\Theta^{\eps}\colon (\rho,u)\mapsto \left(\rho,\rho u,\rho u\ocircle u,\frac{1}{\eps}\rho^\gamma+\rho\frac{|u|^2}{d}\right)$,
	\item $C_{\bar{\rho}}^{\eps}\colon (\rho,u)\mapsto \left(\rho,\rho u,\rho u\ocircle u,\frac{1}{\eps}(\rho^\gamma-{\bar{\rho}}^{\gamma})+\rho\frac{|u|^2}{d}\right)$,
	\item $T_{\bar{\rho}}^{\eps}\colon (\rho,u)\mapsto \left(\rho,u,\frac{1}{\eps \bar{\rho}}(\rho^{\gamma}-\bar{\rho}^{\gamma})\right)$,
	\item $P_{\bar{\rho}}^{\eps}\colon (\rho,u)\mapsto \left(u,\frac{1}{\eps \bar{\rho}}(\rho^\gamma-\bar{\rho}^{\gamma})\right)$,
\end{itemize}
where $\bar{\rho}>0$ will denote the constant density of the incompressible limit system. Here, we defined $v\ocircle v:=v\otimes v-\frac{|v|^2}{d}\mathbb{E}_d$, where $(v\otimes w)_{i,j}=v_iw_j$ is the rank-one matrix corresponding to $v,w\in \R^d$. The letters $S$, $C$, and $T$ have no intentional meaning in the above definition, the notation $\Theta$ comes from \cite{GW21}, and the map $P$ in the fifth bullet point stands for the \textit{pressure lift}.\\
In the following, the (space) dimension will be denoted by $d\in\N$. Further, set $N:=\dim(\R^+\times \R^d\times S_0^d\times \R)$, which is the dimension of the state space of the relaxed Euler system (\ref{eq:relaxedeulersystem}), see below. Here, $S_0^d$ is the space of symmetric trace-free $(d\times d)$-matrices. Also, $T>0$ will always denote some fixed time.\\
\\
Denote by $\mathcal{P}(\R^m)$ the set of probability measures over $\R^m$ and denote by $L_{\operatorname{w}}^{\infty}\left(\Omega,\mathcal{P}(\R^m)\right)$ the set of weakly*-measurable maps from an open set $\Omega\subset \R^d$ into the probability measures. Elements of $L_{\operatorname{w}}^{\infty}\left(\Omega,\mathcal{P}(\R^m)\right)$ will be called Young measures. Let us also introduce the following notion of convergence tailored to Young measures:\\
A sequence $(z_n)$ of measurable functions generates a Young measure $\nu$ if
\begin{align*}
	f(z_n(\bullet))\overset{*}{\rightharpoonup}\int\limits_{\R^m}^{}f(z)\dd\nu_{(\bullet)}(z)=:\left\langle\nu_{(\bullet)},f\right\rangle\text{ in }L^{\infty}(\Omega)
\end{align*}
for all $f\in C_0(\R^m)$. We will denote this by $z_n\overset{Y}{\rightharpoonup}\nu$. Here, $C_0(\R^m)$ denotes the space of all continuous functions that vanish at infinity, which is obtained by taking the closure of the space of compactly supported continuous functions on $\R^m$ with respect to the supremum norm.\\
Similarly, we say that a sequence of Young measures $\nu_n$ converges to $\nu$ in the sense of Young measures if
\begin{align*}
	\langle\nu_n,f\rangle\overset{*}{\rightharpoonup}\langle\nu,f\rangle\text{ in }L^{\infty}(\Omega)
\end{align*}
for all $f\in C_0(\R^m)$. This is again denoted by $\nu_n\overset{Y}{\rightharpoonup}\nu$.
\subsection*{Measure-valued solutions of the isentropic Euler system}
Let $\gamma>1$ be the adiabatic exponent. We start by clarifying the notion of measure-valued solution corresponding to the isentropic Euler system (\ref{eq:isentropiceulereps}).
\begin{defi}
	We call a Young measure $\nu\in L^{\infty}_{\operatorname{w}}\left((0,T)\times \T^d,\mathcal{P}(\R^+\times \R^d)\right)$ a measure-valued solution of (\ref{eq:isentropiceulereps}) with Mach number $\sqrt{\eps}$ and initial data $(\rho_0,u_0)\in L^1(\T^d)$ if
	\begin{equation}
		\begin{aligned}
			\int\limits_{0}^{T}\int\limits_{\T^d}^{}\partial_t\varphi\cdot \langle\nu,\rho u\rangle+\nabla\varphi\,:\,\langle\nu,\rho u\otimes u\rangle+\left\langle\nu,\frac{1}{\eps}\rho^{\gamma}\right\rangle\operatorname{div}\varphi\,\dx\dt&=-\int\limits_{\T^d}^{}\varphi(0,\bullet)\cdot\rho_0u_0\,\dx,\\
			\int\limits_{0}^{T}\int\limits_{\T^d}^{}\langle\nu,\rho\rangle\partial_t\psi+\nabla\psi\cdot \langle\nu,\rho u\rangle\dx\dt&=-\int\limits_{\T^d}^{}\rho_0\psi(0,\bullet)\dx\label{eq:mvsisentropic}
		\end{aligned}
	\end{equation}
	for all $\varphi\in C_c^{\infty}\left([0,T)\times \T^d,\R^d\right)$ and $\psi\in C_c^{\infty}\left([0,T)\times \T^d\right)$. Note that the above integrals are assumed to exist as part of the definition.\\
	If $\nu=\delta_{(\rho,u)}$ for some $(\rho,u)\in L^{1}\left((0,T)\times \T^d,\R^+\times \R^d\right)$ we call $\nu$, or simply $(\rho,u)$, a weak solution of (\ref{eq:isentropiceulereps}).
\end{defi}
Define the energy density function corresponding to the Mach number $\sqrt{\eps}$ by
\begin{align*}
	e_{\eps}\colon (\rho,u)\mapsto \frac{1}{2}\rho|u|^2+\frac{1}{\gamma-1}\frac{1}{\eps}\rho^{\gamma}.
\end{align*}
\begin{defi}
	Let $\nu$ be a measure-valued solution of (\ref{eq:isentropiceulereps}) with initial data $(\rho_0,u_0)\in L^1(\T^d)$. We call $\nu$ energy admissible if
	\begin{align*}
		\underset{t\in(0,T)}{\esssup}\int\limits_{\T^d}^{}\left\langle\nu_{(t,x)},e_{\eps}\right\rangle\dx\leq \int\limits_{\T^d}^{}e_{\eps}(\rho_0(x),u_0(x))\dx<\infty.
	\end{align*}
\end{defi}
\subsection*{Measure-valued solutions of the incompressible Euler system}
For the case of the incompressible Euler system we give two different ways to define measure-valued solutions depending on the space of test functions under consideration. The classical notion of measure-valued solution has been first introduced by DiPerna, cf. \cite{D85}. Note, that we constrain ourselves to oscillation measures and also assume a certain boundedness in the $u$ variable of our measures.
\begin{defi}\label{defi:mvsdefinitions}
	\begin{enumerate}
		\item[(i)] \label{a}A Young measure $\nu\in L^{\infty}_{\operatorname{w}}\left((0,T)\times \T^d,\mathcal{P}(\R^d)\right)$ is called a classical measure-valued solution of (\ref{eq:incompressibleeuler}) with initial data $u_0\in L^1(\T^d)$ if $\langle\nu,|u|^2\rangle\in L^{\infty}\left((0,T)\times \T^d\right)$ and
		\begin{align*}
			\int\limits_{0}^{T}\int\limits_{\T^d}^{}\partial_t\varphi\cdot\langle\nu,u\rangle+\nabla\varphi:\langle\nu,u\otimes u\rangle\dx\dt&=-\int\limits_{\T^d}^{}\varphi(0)\cdot u_0\,\dx,\\
			\int\limits_{\T^d}^{}\nabla\psi\cdot\langle\nu,u\rangle\dx&=0
		\end{align*}
		for all $\varphi\in C_c^{\infty}\left([0,T)\times \T^d,\R^d\right)$ with $\operatorname{div}\varphi=0$ and for all $\psi\in C_c^{\infty}(\T^d)$.
		\item[(ii)] \label{b}A Young measure $\mu\in L^{\infty}_{\operatorname{w}}\left((0,T)\times \T^d,\mathcal{P}(\R^d\times \R)\right)$ is called an augmented measure-valued solution of (\ref{eq:incompressibleeuler}) with initial data $u_0\in L^1(\T^d)$ if $\langle\mu,|u|^2\rangle\in L^{\infty}\left((0,T)\times \T^d\right)$ and
		\begin{align*}
			\int\limits_{0}^{T}\int\limits_{\T^d}^{}\partial_t\varphi\cdot\langle\mu,u\rangle+\nabla\varphi:\langle\mu,u\otimes u\rangle+\langle\mu,P\rangle\operatorname{div}\varphi\,\dx\dt&=-\int\limits_{\T^d}^{}\varphi(0)\cdot u_0\,\dx,\\
			\int\limits_{\T^d}^{}\nabla\psi\cdot\langle\mu,u\rangle\dx&=0
		\end{align*}
		for all $\varphi\in C_c^{\infty}\left([0,T)\times \T^d,\R^d\right)$ and for all $\psi\in C_c^{\infty}(\T^d)$.\\
		We use the word ``augmented'' here, since $\mu$ lives on the state space $\left\{(u,P)\in \R^d\times \R\right\}$, which is augmented by the pressure variable compared to the state space $\left\{u\in \R^d\right\}$ of the classical measure-valued solutions.
	\end{enumerate}
\end{defi}
It can be shown that $\langle\nu,u\rangle$ is weakly continuous in time if $\nu$ is a measure-valued solution, cf. Appendix A in \cite{DS10}. Hence, $u_0$ is necessarily divergence-free.\\
As for the isentropic Euler system, measure-valued solutions $\nu=\delta_{(u,P)}$ concentrated on a single pair of functions $(u,P)\in L^{1}\left((0,T)\times \T^d,\R^d\times \R\right)$ are called \textit{weak solutions} of (\ref{eq:incompressibleeuler}). Note that for weak solutions the definitions of classical and augmented measure-valued solutions are equivalent. Therefore, it suffices to consider only the velocity $u$. The corresponding average-free pressure is then obtained from the formula $P=(-\Delta)^{-1}\operatorname{div}\operatorname{div}(u\otimes u)$. In this case we also simply say that $u$ is a weak solution of (\ref{eq:incompressibleeuler}).\\
Let us now compare the two notions of solution in Definition~\ref{defi:mvsdefinitions}.
\begin{lem}\label{lem:mvsdefinitionscomparison}
	If $\nu$ is a classical measure-valued solution, then every Young measure $\mu$ on $\R^d\times \R$ extending $\nu$ in the sense that
	\begin{align*}
		\langle \operatorname{pr}_u(\mu),f\rangle:=\langle \mu,f\circ \operatorname{pr}_u\rangle\equiv \langle\nu,f\rangle\text{ for all }f\in C(\R^d)
	\end{align*}
	is an augmented measure-valued solution if
	\begin{align}
		\langle\mu,P\rangle=(-\Delta)^{-1}\operatorname{div}\operatorname{div}\langle\nu,u\otimes u\rangle.\label{eq:pressurelagrangemultiplieridentity}
	\end{align}
	In particular, the Young measure $\nu\otimes \delta_{(-\Delta)^{-1}\operatorname{div}\operatorname{div}\langle\nu, u\otimes u\rangle}$ is an augmented solution.\\
	On the other hand, if $\mu$ is an augmented measure-valued solution, then $\operatorname{pr}_u(\mu)$ is a classical measure-valued solution.
\end{lem}
\begin{proof}
	Let $\nu$ be a classical measure-valued solution. Then every extension of $\nu$ satisfying (\ref{eq:pressurelagrangemultiplieridentity}) is an augmented solution. This follows by standard techniques using the Helmholtz decomposition of the space of test functions. We also note here that due to a Calder\'on-Zygmund argument one can see from (\ref{eq:pressurelagrangemultiplieridentity}) that $\langle\mu,P\rangle\in L^1\left((0,T)\times \T^d\right)$ since $\langle\nu, u\otimes u\rangle\in L^{\infty}\left((0,T)\times \T^d\right)$. Moreover, the Young measure $\nu\otimes \delta_{(-\Delta)^{-1}\operatorname{div}\operatorname{div}\langle\nu, u\otimes u\rangle}$ clearly extends $\nu$ and fulfills
	\begin{align*}
		\left\langle \nu\otimes \delta_{(-\Delta)^{-1}\operatorname{div}\operatorname{div}\langle\nu, u\otimes u\rangle},P\right\rangle=\left\langle \delta_{(-\Delta)^{-1}\operatorname{div}\operatorname{div}\langle\nu, u\otimes u\rangle},P\right\rangle=(-\Delta)^{-1}\operatorname{div}\operatorname{div}\langle\nu, u\otimes u\rangle.
	\end{align*}
	This proves the first part of the lemma.\\
	Now let $\mu$ be an augmented measure-valued solution, then testing against divergence-free functions yields that $\operatorname{pr}_u(\mu)$ is a classical measure-valued solution.
\end{proof}
\begin{rem}
	The above lemma implies that there is an injection of the space of classical measure-valued solutions into the space of augmented measure-valued solutions by identifying those via $\nu\mapsto \nu\otimes\delta_{(-\Delta)^{-1}\operatorname{div}\operatorname{div}\langle\nu, u\otimes u\rangle}$. Hence, augmented measure-valued solutions are a more general solution concept. Note here that an augmented measure-valued solution determines the corresponding classical measure-valued solution completely, while in the other direction by (\ref{eq:pressurelagrangemultiplieridentity}) only the first moment in the $P$ variable is fixed for the augmented solution. Thus, in a sense, augmented solutions contain more information than classical measure-valued solutions. Especially, we will be able to formulate the low Mach limit process within the framework of augmented solutions.\\
\end{rem}
\begin{rem}
	In particular, from Lemma~\ref{lem:mvsdefinitionscomparison} we infer that an augmented measure-valued solution $\mu$ can be identified with a classical measure-valued solution if $\mu=\nu\otimes \delta_{(-\Delta)^{-1}\operatorname{div}\operatorname{div}\langle \nu, u\otimes u\rangle}$ for some Young measure $\nu$. Note that this condition is non-local. On the other hand, if one can guarantee that the augmented solution $\mu$ is atomic in the pressure variable, it necessarily has to be of this form.
\end{rem}
Since the internal pressure does not contribute to the total energy in the incompressible case, we define the energy density function by
\begin{align*}
	e\colon u\mapsto \frac{1}{2}|u|^2.
\end{align*}
In the following definition of energy admissibility it suffices to treat only the case of classical measure-valued solutions, as two augmented solutions corresponding to the same classical measure-valued solution differ only in the pressure component
.
\begin{defi}
	Let $\nu$ be a measure-valued solution with initial data $u_0\in L^2(\T^d)$. We call $\nu$ energy admissible if
	\begin{align*}
		\underset{t\in(0,T)}{\esssup}\int\limits_{\T^d}^{}\left\langle\nu_{(t,x)},e\right\rangle\dx\leq \int\limits_{\T^d}^{}e(u_0(x))\dx.
	\end{align*}
\end{defi}
Let us also observe an important weak-strong uniqueness property of augmented solutions.
\begin{prop}\label{prop:wsuniqueness}
	Let $\nu$ be an energy admissible augmented measure-valued solution and let $(U,\Pi)\in C^1\left([0,T]\times \T^d\right)$ be a strong solution of (\ref{eq:incompressibleeuler}). Assume that both $\nu$ and $(U,\Pi)$ arise from the same initial datum $u_0$. Then $\nu=\delta_U\otimes \mu$ for some $\mu\in L_{\operatorname{w}}^{\infty}\left((0,T)\times \T^d,\R\right)$ acting on the space of pressures.
\end{prop}
\begin{proof}
	We need to show that $\supp\left(\nu_{(t,x)} \right)\subset \{(u,P)\,:\,u=U(t,x) \}$ almost everywhere. For that, we follow the proof of a standard weak-strong uniqueness argument for classical measure-valued solutions, cf. for example Theorem~3.4 in \cite{W18}. The proof will be given here for convenience.\\
	Define the relative energy $E_{\operatorname{rel}}(t):=\frac{1}{2}\int\limits_{\T^d}^{}\left\langle\nu_{(t,x)},|u-U(t,x)|^2\right\rangle\dx$ for almost every $t\in (0,T)$. We estimate
	\begin{align*}
		E_{\operatorname{rel}}(t)&=\frac{1}{2}\int\limits_{\T^d}^{}|U(t,x)|^2+\left\langle\nu_{(t,x)},|u|^2\right\rangle-2\left\langle \nu_{(t,x)},u\right\rangle\cdot U(t,x)\dx\\
		&\leq -\int\limits_{0}^{t}\int\limits_{\T^d}^{}\partial_t U(\tau,x)\cdot\left\langle\nu_{(\tau,x)},u\right\rangle+\nabla U(\tau,x) :\left\langle\nu_{(\tau,x)},u\otimes u\right\rangle+\left\langle\nu_{(\tau,x)},P\right\rangle\operatorname{div} U(\tau,x)\dx\dd\tau\\
		&=\int\limits_{0}^{t}\int\limits_{\T^d}^{}\big(\operatorname{div}(U\otimes U)(\tau,x)+\nabla \Pi(\tau,x)\big)\cdot\left\langle\nu_{(\tau,x)},u\right\rangle-\nabla U(\tau,x) :\left\langle\nu_{(\tau,x)},u\otimes u\right\rangle\dx\dd\tau\\
		&=\int\limits_{0}^{t}\int\limits_{\T^d}^{}\left\langle \nu_{(\tau,x)},(U(\tau,x)-u)\cdot\nabla U(\tau,x)(u-U(\tau,x))\right\rangle\dx\dd\tau\\
		&\leq 2\int\limits_{0}^{t}\|\nabla U(\tau,\cdot)\|_{L^{\infty}(\T^d)}E_{\operatorname{rel}(\tau)}\dd\tau.
	\end{align*}
	Here, we used in the second line the energy (in)equality for $\nu$ and $(U,\Pi)$. In the fourth line we used the identities $\operatorname{div}(U\otimes U)=U\cdot\nabla U$ and $\nabla U : (u\otimes u)=u\cdot \nabla U\cdot u$. Moreover, the fact that $\int\limits_{\T^d}^{}\left\langle\nu_{(\tau,x)},(U(\tau,x)-u)\right\rangle\cdot (U\cdot\nabla U)\dx=0$ holds for almost every $\tau$ has also been applied. By Grönwall's inequality we infer that $E_{\operatorname{rel}}(t)=0$ for almost every $t$ and hence $\nu$ is atomic in the $u$ variable.
\end{proof}
\begin{rem}
	The weak-strong uniqueness result above shows that the velocity component of $\nu$ is fixed and equal to the strong solution's velocity. On the other hand the pressure variable cannot be uniquely determined. For that observe that if $(U,\Pi)$ is a strong solution, then for example both $\delta_U\otimes \delta_{\Pi}$ and $\delta_U\otimes \left(\frac{1}{2}\delta_{4\Pi}+\frac{1}{2}\delta_{-2\Pi} \right)$ are energy admissible augmented measure-valued solutions to the same initial data. In fact, the latter measures are also admissible with respect to the local energy inequality. So, there exists an inherent non-uniqueness, at least in the pressure, for augmented solutions.
\end{rem}
We formulate the incompressible Navier-Stokes system with viscosity parameter $\alpha>0$ as
\begin{align*}
	\partial_tu+\operatorname{div}(u\otimes u)+\nabla P&=\alpha\Delta u,\\
	\operatorname{div}u&=0.
\end{align*}
Weak solutions of this system may be defined as in Definition~\ref{defi:mvsdefinitions}. In this case also for the Navier-Stokes system the pressure is already fully determined by the velocity. If $(\alpha_n)$ is a null sequence of viscosity parameters and if $u_n$ is a weak solution of the incompressible Navier-Stokes system corresponding to $\alpha_n$ for all $n\in \N$, then $u_n$ is called a \textit{vanishing viscosity sequence}. Such a sequence is said to be energy admissible if for all $n\in \N$
\begin{align*}
	\underset{t\in(0,T)}{\esssup}\left(\int\limits_{\T^d}^{}e(u_n(t,x))\dx+\int\limits_{0}^{t}\int\limits_{\T^d}^{}\alpha|\nabla u_n(\tau,x)|^2\dd\tau\dx\right)\leq \int\limits_{\T^d}^{}e(u_n^0(x))\dx<\infty
\end{align*}
with $e(u)=\frac{1}{2}|u|^2$ as for the incompressible Euler equations, where $u_n^0$ is the initial data corresponding to $u_n$.
\subsection*{The relaxed Euler system and homogeneous differential operators}
As an auxiliary tool we consider the relaxed Euler system
\begin{equation}
	\begin{aligned}
		\partial_t m+\operatorname{div}M+\nabla Q&=0,\\
		\partial_t \rho+\operatorname{div}m&=0\label{eq:relaxedeulersystem}
	\end{aligned}
\end{equation}
for the variables $(\rho,m,M,Q)\in \R\times \R^d\times S_0^d\times \R$. This can be viewed as a linearized version of the Euler equations, both isentropic or incompressible, extended by additional variables which substitute the nonlinear terms. One can easily check that if $(\rho,u)$ is a weak solution of (\ref{eq:isentropiceulereps}) with Mach number $\sqrt{\eps}$ then $\Theta^{\eps}(\rho,u)$ is a distributional solution of (\ref{eq:relaxedeulersystem}). Note that the relaxed Euler system does not depend on the Mach number $\sqrt{\eps}$, since we simply replace the whole nonlinear generalized pressure term $\frac{1}{\eps}\rho^\gamma+\rho\frac{|u|^2}{d}$ by the variable $Q$. Similarly, if $u$ is a weak solution of (\ref{eq:incompressibleeuler}) with corresponding pressure $P$, the lift $S(u,P)$ solves (\ref{eq:relaxedeulersystem}) also in the sense of distributions.\\
\\
The relaxed Euler system (\ref{eq:relaxedeulersystem}) may be reformulated in a more abstract sense as a linear homogeneous differential operator $\mathcal{A}_E$ of order one with constant coefficients. Here, the subscript indicates the correspondence to the Euler system.\\
In general, a \textit{linear homogeneous differential operator} is an operator $\mathcal{A}$ of the form $\mathcal{A}=\sum\limits_{|\alpha|=k}^{}A^{\alpha}\partial_{\alpha}$, where $A^{\alpha}$ are constant coefficient matrices acting on $\R^m$ and $k$ denotes the order of homogeneity of $\mathcal{A}$. Following \cite{R18}, we say that a linear homogeneous differential operator $\mathcal{B}$ over $\R^d$ is a \textit{potential} for $\mathcal{A}$ if
\begin{align*}
	\ker\mathbb{A}(\xi)=\image\mathbb{B}(\xi)
\end{align*}
for all $0\neq \xi\in \R^d$, where $\mathbb{A}$ and $\mathbb{B}$ are the corresponding Fourier symbols to $\mathcal{A}$ and $\mathcal{B}$, respectively. Furthermore, we say that two vectors $z_1,z_2\in \R^m$ are $\mathcal{A}$\textit{-wave-cone-connected} if $(z_1-z_2)\in \underset{|\xi|=1}{\bigcup}\ker \mathbb{A}(\xi)$.\\
For the rest of the paper we assume that $\mathcal{A}_E$ admits a potential $\mathcal{B}_E$ over $\R^{d+1}$ of order two. This fact has been shown in Lemma~5.6 in \cite{GW21} in the case $d=2$ and is presumably also correct in higher dimensions. It is straightforward to check that $\mathcal{A}_E$ has constant rank, i.e. the rank of its Fourier symbol $\mathbb{A}_E(\xi)$ is constant for all $\xi\neq 0$, cf. Lemma~1 in \cite{CFKW15}.\\
\\
The following concepts have been introduced in \cite{FM99}.
\begin{defi}
	For a linear homogeneous differential operator $\mathcal{A}$ and for $z\in\R^m$ define the $\mathcal{A}$\textit{-quasiconvex envelope} for $f\in C(\R^m)$ on $\T^d$ by
	\begin{align*}
		Q_{\mathcal{A}}f(z):=\inf\left\{\underset{\T^d}{\dashint}f(z+\varphi(x))\dx\,:\,\varphi\in C^{\infty}(\T^d)\cap \ker(\mathcal{A}),\ \underset{\T^d}{\dashint}\varphi\,\dx=0 \right\}.
	\end{align*}
	and for $q>0$ define a truncated version of that by
	\begin{align*}
		Q_{\mathcal{A}}^qf(z):=\inf\left\{\underset{\T^d}{\dashint}f(z+\varphi(x))\dx\,:\,\varphi\in C^{\infty}(\T^d)\cap \ker(\mathcal{A}),\ \underset{\T^d}{\dashint}\varphi\,\dx=0,\textit{ and }\|\varphi\|_{L^{\infty}(\T^d)}\leq q \right\}.
	\end{align*}
\end{defi}
If $\mathcal{A}$ admits a potential operator $\mathcal{B}$, the $\mathcal{A}$-quasiconvex envelope can be rewritten for all $z\in\R^m$ and $f\in C(\R^m)$ as
\begin{align*}
	Q_{\mathcal{A}}f=\inf\left\{\int\limits_{(0,1)^d}^{}f(z+\mathcal{B}w(x))\dx\,:\,w\in C_c^{\infty}\big((0,1)^d\big) \right\},
\end{align*}
cf. Corollary~1 in \cite{R18}. Inspired by this, a truncated version of the potential formulation of the quasiconvex envelope has been defined in \cite{GW21} by
\begin{align*}
	Q_{\mathcal{B}}^qf(z):=\inf\left\{\int\limits_{(0,1)^d}^{}f(z+\mathcal{B}w(x))\dx\,:\,w\in C_c^{\infty}\big((0,1)^d\big),\ \|D^lw\|_{L^{\infty}\left((0,1)^d\right)}\leq q \right\}
\end{align*}
for $q>0$, $z\in \R^m$, and $f\in C(\R^m)$, where $l$ is the order of $\mathcal{B}$. It follows directly from the definition that for all $f\in C(\R^m)$ and $q>0$ we have
\begin{align*}
	 Q_{\mathcal{A}}f\leq Q_{\mathcal{A}}^qf\leq Q_{\mathcal{B}}^{C_{\mathcal{B}}q}f\leq f
\end{align*}
for some constant $C_{\mathcal{B}}$ depending only on the coefficient matrices of $\mathcal{B}$.

\section{Limits of low Mach sequences}\label{sect:limitsoflowmachsequences}
We begin our investigation of low Mach limits by the observation that under certain assumptions on a low Mach sequence of measure-valued solutions the limit 
is an augmented measure-valued solution of the incompressible Euler system. 
We also specify in what sense a low Mach limit measure is generated.
\begin{theo}\label{theo:measuresequenceimpliessolution}
	Let $(\nu^{\eps})_{\eps>0}$ be a family of Young measures with the following properties:
	\begin{itemize}
		\item The measures $\nu^{\eps}$ solve (\ref{eq:mvsisentropic}) with initial data $(\rho_0^{\eps},u_0^{\eps})\in L^1(\T^d)$.
		\item There exist $R>\eta>0$ and a compact set $K\subset \R^d\times \R$ such that $\supp\left((T_{\bar{\rho}}^{\eps})_{\#}\nu^{\eps}\right)\subset [\eta,R]\times K$.
		\item As $\eps$ goes to zero, suppose that the initial data $(\rho^{\eps}_0,u^{\eps}_0)$ converge weakly in $L^1$ to $(\bar{\rho},u_0)$, where $0<\bar{\rho}\equiv \text{const.}$ and $\operatorname{div}u_0=0$.
	\end{itemize}
	Then there exists a subsequence $\nu^{\eps_k}$ such that the pressure lifts $(P_{\bar{\rho}}^{\eps_k})_{\#}\nu^{\eps_k}$ generate an augmented measure-valued solution $\mu$ of (\ref{eq:incompressibleeuler}) with initial data $u_0$.\\
	If the measure-valued solutions $\nu^{\eps}$ are also energy admissible and the initial data $(\rho^{\eps}_0,u^{\eps}_0)$ converge strongly in $L^{\gamma}\times L^2$. Then additionally we obtain that $\mu$ is energy admissible.
\end{theo}
In the situation of Theorem~\ref{theo:measuresequenceimpliessolution}, namely $(P^{\eps}_{\bar{\rho}})_{\#}\nu^{\eps}\overset{Y}{\rightharpoonup}\mu$, we say that $\mu$ is the \textit{low Mach limit} of the sequence $\nu^{\eps}$.
\begin{proof}
	The functions $\langle\nu^{\eps},\operatorname{id}\rangle\in L^{\infty}\left([0,T]\times \T^d\right)$ are weakly continuous in time by a standard argument, see Appendix~A in \cite{DS10}. Moreover, the sequence $\langle\nu^{\eps},\operatorname{id}\rangle\subset L^{\infty}$ is uniformly bounded. Hence, the sequence of initial data $(\rho^{\eps}_0,u^{\eps}_0)\subset L^{\infty}(\T^d)$ is also uniformly bounded.\\
	Since $(T_{\bar{\rho}}^{\eps})_{\#}\nu^{\eps}$ has uniformly bounded support, by Lemma~4.3 in \cite{Rindler} there exists a subsequence $(T_{\bar{\rho}}^{\eps_k})_{\#}\nu^{\eps_k}$ and a Young measure $\chi$ such that $(T_{\bar{\rho}}^{\eps_k})_{\#}\nu^{\eps_k}\overset{Y}{\rightharpoonup}\chi$.\\
	\\
	\textbf{Claim 1:} The measure $\chi$ fulfills $\supp\left(\chi_{(t,x)}\right)\subset \{(\rho,u,P)\,:\,\rho=\bar{\rho}\}$ for a.e.~$(t,x)\in(0,T)\times \T^d$.\\
	\\
	Indeed, define $U_n:=\left\{(\rho,u,P)\,:\,|\rho-\bar{\rho}|>\frac{1}{n}\right\}$ for all $n\in\N$. Fix $n\in\N$ and choose a function $f\in \operatorname{Lip}\left(\R\times \R^d\times \R,[0,1]\right)$ with $f\equiv 1$ on $U_{n}$ and $f\equiv 0$ on $\{(\rho,u,P)\,:\,\rho=\bar{\rho}\}$ and Lipschitz constant $L_f$. Let $E\subset (0,T)\times \T^d$ be measurable. We estimate
	\begin{align*}
		0&\leq \int\int \chi_{}({U_{n}})\mathds{1}_E\,\dx\dt\\
		&\leq\int\int\langle \chi,|f(\rho,u,P)-f(\bar{\rho},u,P)|\rangle\mathds{1}_{E}\,\dx\dt\\
		&\leq \int\int\langle \chi,L_f|\rho-\bar{\rho}|\rangle\mathds{1}_{E}\,\dx\dt\\
		&=\lim\limits_{k\rightarrow\infty}L_f\int\int\left\langle (T_{\bar{\rho}}^{\eps_k})_{\sharp}\nu^{\eps_k},|\rho-\bar{\rho}|\right\rangle\mathds{1}_{E}\,\dx\dt\\
		&=\lim\limits_{k\rightarrow\infty}L_f\int\int\left\langle \nu^{\eps_k},|\rho-\bar{\rho}|\right\rangle\mathds{1}_{E}\,\dx\dt\\
		&\leq \liminf\limits_{k\rightarrow \infty}L_f\left(\int\int\langle\nu^{\eps_k},|\rho-\bar{\rho}|^2\rangle\dx\dt \right)^{\frac{1}{2}}\cdot \|\mathds{1}_{E}\|_{L^2((0,T)\times \T^d)}\\
		&\leq \liminf\limits_{k\rightarrow \infty}C\cdot L_f\sqrt{\eps_k}\left(\int\int\left\langle\nu^{\eps_k}, \frac{1}{\eps_k}\left(\rho^{\gamma}-\gamma\bar{\rho}^{\gamma-1}(\rho-\bar{\rho})-\bar{\rho}^{\gamma}\right)\right\rangle\dx\dt \right)^{\frac{1}{2}}\cdot \sqrt{|\T^d|\cdot T}\\
		&\leq \liminf\limits_{k\rightarrow \infty}C\cdot L_f\sqrt{\eps_k}\left(\int\int \left\langle\nu^{\eps_k},\frac{1}{\eps_k}\left|\rho^{\gamma}-\bar{\rho}^{\gamma}\right|\right\rangle\dx\dt \right)^{\frac{1}{2}}\\
		&=0,
	\end{align*}
	where we used the strong convexity of $\rho\mapsto\rho^{\gamma}$ and that the support of $\nu^{\eps}$ is uniformly bounded away from zero in the $\rho$ variable. We also used that $\int\left\langle\nu^{\eps_k},\rho-\bar{\rho}\right\rangle\dx\,(t)=\int\rho^{\eps_k}_0-\bar{\rho}\dx\rightarrow0$ for almost every $t$. The above estimate implies that $0=\chi_{}(U_{n})$ almost everywhere for all $n$. 
	As $\{(\rho,u,P)\,:\,\rho\neq\bar{\rho}\}$ is open, this implies that $\{(\rho,u,P)\,:\,\rho\neq\bar{\rho}\}\subset \supp\left(\chi_{(t,x)}\right)^c$ for a.e.~$(t,x)\in (0,T)\times \T^d$.
	\\
	\\
	\textbf{Claim 2:} The measure $\mu$, defined by $\langle \mu,f\rangle:=\left\langle \chi,f\circ \operatorname{pr}_{(u,P)}\right\rangle$ for all $f\in C(\R^d\times \R)$, is an augmented measure-valued solution of (\ref{eq:incompressibleeuler}).\\
	\\
	Indeed, from Claim 1 we infer that
	\begin{align*}
		\langle\nu^{\eps_k},\rho\rangle=\left\langle(T_{\bar{\rho}}^{\eps_k})_{\#}\nu^{\eps_k},\rho\right\rangle&\overset{*}{\rightharpoonup}\langle\chi,\rho\rangle=\bar{\rho}\text{ in }L^{\infty},\\
		\langle\nu^{\eps_k},\rho u\rangle=\left\langle(T_{\bar{\rho}}^{\eps_k})_{\#}\nu^{\eps_k},\rho u\right\rangle&\overset{*}{\rightharpoonup}\langle\chi,\rho u\rangle=\bar{\rho}\langle\mu,u\rangle\text{ in }L^{\infty},\\
		\langle\nu^{\eps_k},\rho u\otimes u\rangle=\left\langle(T_{\bar{\rho}}^{\eps_k})_{\#}\nu^{\eps_k},\rho u\otimes u\right\rangle&\overset{*}{\rightharpoonup}\langle\chi,\rho u\otimes u\rangle=\bar{\rho}\langle\mu, u\otimes u\rangle\text{ in }L^{\infty},\\
		\left\langle\nu^{\eps_k},\frac{1}{\eps_k}\rho^{\gamma}\right\rangle-\frac{\bar{\rho}^{\gamma}}{\eps_k}=\bar{\rho}\left\langle(T_{\bar{\rho}}^{\eps_k})_{\#}\nu^{\eps_k},\operatorname{pr}_{P}\right\rangle&\overset{*}{\rightharpoonup}\bar{\rho}\langle\chi,P\rangle=\bar{\rho}\langle\mu, P\rangle\text{ in }L^{\infty}.
	\end{align*}
	Thus, 
	$\mu$ is a measure-valued solution of the incompressible Euler equations, since the initial data converge weakly.\\
	\\
	Note that for all $f\in C_0\left(\R^d\times \R\right)$ we have
	\begin{align*}
		\left\langle(P_{\bar{\rho}}^{\eps_k})_{\#}{\nu}^{\eps_k},f\right\rangle=
		\left\langle\nu^{\eps_k}, f\circ \operatorname{pr}_{(u,P)}\circ T_{\bar{\rho}}^{\eps_k}\right\rangle 
		\overset{*}{\rightharpoonup}\left\langle\chi,f\circ \operatorname{pr}_{(u,P)}\right\rangle=\langle\mu, f\rangle.
	\end{align*}
	Thus, $(P_{\bar{\rho}}^{\eps_k})_{\#}{\nu}^{\eps_k}$ converges to $\mu$ in the sense of Young measures.\\
	\\
	Let us now additionally assume the energy admissibility and strong convergence of the initial data as stated in the second part of the theorem. It only remains to check the energy inequality for $\mu$. For that let $\varphi\in C_c^{\infty}\left([0,T),[0,\infty)\right)$. We then have
	\begin{align*}
		\int\limits_{0}^{T}\varphi(t)\left(\int\limits_{\T^d}^{}\left\langle(\nu^{\eps_k})_{(t,x)},\frac{1}{2}\rho|u|^2+\frac{1}{\eps_k}\frac{1}{\gamma-1}\left(\rho^{\gamma}-\bar{\rho}^{\gamma}\right)\right\rangle\dx \right)\dt\leq \int\limits_{0}^{T}\varphi(t)\dt\int\limits_{\T^d}^{}\frac{1}{2}\rho^{\eps_k}_0|u^{\eps_k}_0|^2\dx.
	\end{align*}
	The right hand side converges to $\int\limits_{0}^{T}\varphi(t)\dt\int\limits_{\T^d}^{}\frac{1}{2}\bar{\rho}|u_0|^2\dx$ by the strong convergence and uniform boundedness of the initial data $\left(\rho^{\eps_k}_0,u^{\eps_k}_0\right)$. For the left hand side we obtain
	\begin{align*}
		&\int\limits_{0}^{T}\varphi(t)\left(\int\limits_{\T^d}^{}\left\langle(\nu^{\eps_k})_{(t,x)},\frac{1}{2}\rho|u|^2+\frac{1}{\eps_k}\frac{1}{\gamma-1}\left(\rho^{\gamma}-\bar{\rho}^{\gamma}\right)\right\rangle\dx \right)\dt\\
		&=\int\limits_{0}^{T}\varphi(t)\left(\int\limits_{\T^d}^{}\left\langle\left((T_{\bar{\rho}}^{\eps_k})_{\#}\nu^{\eps_k}\right)_{(t,x)},\frac{1}{2}\rho|u|^2+\frac{\bar{\rho}}{\gamma-1}\operatorname{pr}_P\right\rangle\dx \right)\dt\\
		&\rightarrow \int\limits_{0}^{T}\varphi(t)\left(\int\limits_{\T^d}^{}\bar{\rho}\frac{1}{2}\langle\mu,|u|^2\rangle+\bar{\rho}\frac{1}{\gamma-1}\langle\mu,P\rangle\dx \right)\dt.
	\end{align*}
	Now, as $\mu$ is an augmented solution of (\ref{eq:incompressibleeuler}), its first moment of the pressure $\langle\mu, P\rangle$ has to fulfill
	\begin{align*}
		\langle\mu,P\rangle=\operatorname{div}\operatorname{div}(-\Delta)^{-1}\langle\mu,u\otimes u\rangle.
	\end{align*}
	Moreover, since $\mu$ has compact support, we have $\langle\mu,u\otimes u\rangle\in L^2((0,T)\times \T^d)$. Thus, standard singular integral arguments yield that $\langle\mu, P\rangle$ is given as the weak divergence of a $W^{1,2}$-function, hence its mean value on the torus $\T^d$ is zero almost everywhere in time. This yields the assertion by choosing $\varphi$ as a non-negative approximate identity.
\end{proof}
\begin{rem}
	In the above proof, we have seen that the initial data $(\rho_0^{\eps},u_0^{\eps})$ are uniformly bounded in $L^{\infty}([0,T]\times \T^d)$. Since we consider a bounded space-time domain, we could have simply required the initial data to converge in $L^1$ instead of $L^{\gamma}\times L^2$ in the energy admissibility part of the statement of Theorem~\ref{theo:measuresequenceimpliessolution}. The same observation applies to Corollary~\ref{cor:sequenceimpliessolution} and Proposition~\ref{prop:characterizationapplication} below.
\end{rem}
If every measure-valued solution of a low Mach sequence admits a generating sequence of weak solutions, then the low Mach limit can be generated by a low Mach sequence of weak solutions.
\begin{cor}\label{cor:sequenceimpliessolution}
	Let $(\nu^{\eps})_{\eps>0}$ be a family of Young measures with the following properties:
	\begin{itemize}
		\item The measures $\nu^{\eps}$ solve (\ref{eq:mvsisentropic}) with initial data $(\rho_0^{\eps},u_0^{\eps})\in L^1(\T^d)$.
		\item As $\eps$ goes to zero, suppose that the initial data $(\rho^{\eps}_0,u^{\eps}_0)$ converge weakly in $L^1$ to $(\bar{\rho},u_0)$, where $0<\bar{\rho}\equiv \text{const.}$ and $\operatorname{div}u_0=0$.
		\item For each $\eps>0$ there exists a sequence $\left(\rho_n^{\eps},u_n^{\eps}\right)$ of weak solutions of (\ref{eq:isentropiceulereps}) generating $\nu^{\eps}$ with initial data $\big((\rho_n^{\eps})_0,(u_n^{\eps})_0\big)$.
		\item The sequences $\left(\rho_n^{\eps},u_n^{\eps}\right)$ and $\left(\frac{1}{\eps}\left((\rho_n^{\eps})^{\gamma}-\bar{\rho}^{\gamma}\right)\right)_{\eps,n}$ are uniformly bounded in $\eps,n$ and $\rho_n^{\eps}$ is uniformly bounded away from zero.
	\end{itemize}
	Then there exists a subsequence $(\rho_{n_k}^{\eps_k},u_{n_k}^{\eps_k})$ such that the pressure lifts $P_{\bar{\rho}}^{\eps_k}(\rho_{n_k}^{\eps_k},u_{n_k}^{\eps_k})=\left(u_{n_k}^{\eps_k},\frac{1}{\eps_k\bar{\rho}}\left((\rho_{n_k}^{\eps_k})^{\gamma}-\bar{\rho}^{\gamma}\right)\right)$ generate an augmented measure-valued solution $\mu$ of (\ref{eq:incompressibleeuler}) with initial data $u_0$.\\
	If all $\left(\rho_n^{\eps},u_n^{\eps}\right)$ are also energy admissible and the initial data satisfies
	\begin{align*}
		\big(\left(\rho_n^{\eps}\right)_0,\left(u_n^{\eps}\right)_0\big)\overset{n\rightarrow\infty}{\rightarrow}(\rho^{\eps}_0,u^{\eps}_0)\overset{\eps\rightarrow 0}{\rightarrow}(\bar{\rho},u_0)
	\end{align*}
	 strongly in $L^{\gamma}\times L^2$. Then additionally we obtain that $\mu$ is energy admissible.
\end{cor}
If, as above, it holds that $P^{\eps}_{\bar{\rho}}(\rho^{\eps},u^{\eps})\overset{Y}{\rightharpoonup}\mu$, we say that $\mu$ is the low Mach limit of the sequence $(\rho^{\eps},u^{\eps})$.
\begin{proof}
	The functions $(\rho_n^{\eps},u_n^{\eps})$ and $\langle\nu^{\eps},\operatorname{id}\rangle$ are weakly continuous in time and uniformly bounded. Hence, the initial data $\big((\rho_n^{\eps})_0,(u_n^{\eps})_0\big)$ and $(\rho^{\eps}_0,u^{\eps}_0)$ are uniformly bounded in $L^{\infty}(\T^d)$. Moreover, by testing (\ref{eq:mvsisentropic}) for $(\rho_n^{\eps},u_n^{\eps})$ and $\nu^{\eps}$ against test functions of the form $\varphi(t,x)=\eta(x)\cdot \chi(t)$ one observes that the right hand sides, i.e. the initial data, converge weakly* in $L^{\infty}(\T^d)$ to $(\rho^{\eps}_0,u^{\eps}_0)$ as $n\rightarrow\infty$. Since $L^1(\T^d)$ is separable, there exists a diagonal subsequence such that 
	$\left(\rho_{n_{k}}^{\eps_k}\right)_0\overset{*}{\rightharpoonup} \bar{\rho}$ and $\left(u^{\eps_k}_{n_{k}}\right)_0\overset{*}{\rightharpoonup} u_0$ as $k \rightarrow \infty$.\\
	Therefore, the diagonal sequence $(\rho^{\eps_k}_{n_{k}},u^{\eps_k}_{n_{k}})$ satisfies all requirements of Theorem~\ref{theo:measuresequenceimpliessolution} when we identify these weak solutions with its corresponding Dirac measures. Indeed, since by assumption the generating sequence $T_{\bar{\rho}}^{\eps_k}(\rho^{\eps_k}_{n_k},u^{\eps_k}_{n_k})$ is uniformly bounded and bounded away from zero in the density, we obtain that for all $\eps>0$ it holds that $\supp\left((T_{\bar{\rho}}^{\eps_k})_{\#}\delta_{\left(\rho^{\eps_k}_{n_{k}},u^{\eps_k}_{n_{k}}\right)}\right)\subset [\eta,R]\times K$ for some $R>\eta>0$ and $K\subset \R^d\times \R$ compact. 
	Thus, there exists a subsequence (not relabeled) such that the pressure lifts $P_{\bar{\rho}}^{\eps_k}\left(\rho^{\eps_k}_{n_{k}},u^{\eps_k}_{n_{k}}\right)$ generate an augmented measure-valued solution $\mu$ with initial data $u_0$.\\
	Now additionally assume the energy admissibility and strong convergence of the initial data. Therefore, we could have chosen the above subsequence such that also $\left((\rho_{n_{k}}^{\eps_k})_0,(u_{n_{k}}^{\eps_k})_0\right)\rightarrow (\bar{\rho},u_0)$ in $L^{\gamma}\times L^2$. The second part of Theorem~\ref{theo:measuresequenceimpliessolution} now yields the assertion.
\end{proof}
In the case of monoatomic gases, i.e. $\gamma=1+\frac{2}{d}$, we can also give a version of the previous result with assumptions only on a low Mach sequence of generating measures $\nu^{\eps}$. The main issue here is to obtain suitably bounded compressible weak solutions generating $\nu^{\eps}$.
\begin{prop}\label{prop:characterizationapplication}
	Consider the case $\gamma=1+\frac{2}{d}$ with $d\geq 2$. Let $(\nu^{\eps})_{\eps>0}$ be a family of Young measures with the following properties:
	\begin{itemize}
		\item The measures $\nu^{\eps}$ solve (\ref{eq:mvsisentropic}) with initial data $(\rho_0^{\eps},u_0^{\eps})\in L^1(\T^d)$.
		\item The initial data $(\rho_0^{\eps},u_0^{\eps})$ converge weakly in $L^1$ to $(\bar{\rho},u_0)$ satisfying $0<\bar{\rho}\equiv \text{const.}$ and $\operatorname{div}u_0=0$.
		\item There exists some $R>0$ such that for all $\eps>0$
		\begin{align}
			\left\langle \left((\Theta^{\eps})_\#\nu^{\eps}\right)_{(t,x)},f\right\rangle\geq Q^R_{\mathcal{B}_E}f\left(\left\langle \left((\Theta^{\eps})_\#\nu^{\eps}\right)_{(t,x)},\operatorname{id}\right\rangle\right)\label{eq:sufficientjensencond}
		\end{align}
		for all $f\in C(\R^N)$ and all $(t,x)\in\Omega_0$, where $\Omega_0\subset (0,T)\times \T^d$ is a set of full measure.
		\item For all $\eps>0$ there exist $\sigma^{\eps}\in C\left([0,T]\times \T^d\right)$ and $w^{\eps}\in W^{2,\infty}\left((0,T)\times \T^d\right)$ such that
		\begin{align*}
			\left\langle (\Theta^{\eps})_\#\nu^{\eps},\operatorname{id}\right\rangle=\sigma^{\eps}+\mathcal{B}_Ew^{\eps}.
		\end{align*}
		\item There exists a bounded and continuous function $\chi\colon [0,T]\to [0,\infty)$ such that
		\begin{align*}
			\supp\left(\nu^{\eps}_{(t,x)}\right)\subset \{(\rho,u)\,:\,|\rho-\bar{\rho}|\leq \eps,\ |u|^2\leq \chi(t)\}
		\end{align*}
		holds for all $\eps>0$ and a.e.~$(t,x)\in (0,T)\times \T^d$.
	\end{itemize}
 	Then there exists a subsequence $(\eps_k)_{k\in\N}$ and uniformly bounded weak solutions $(\rho_{k}^{\eps_k},u_{k}^{\eps_k})$ of (\ref{eq:isentropiceulereps}) such that $P_{\bar{\rho}}^{\eps_k}(\rho_{k}^{\eps_k},u_{k}^{\eps_k})\overset{Y}{\rightharpoonup}\mu$, where $\mu$ is an augmented measure-valued solution of (\ref{eq:incompressibleeuler}).\\
 	If we additionally assume that:
 	\begin{itemize}
 		\item The measures $\nu^{\eps}$ are energy admissible and satisfy
 		\begin{align*}
 			\supp\left(\left((\Theta^{\eps})_\#\nu^{\eps}\right)_{(t,x)} \right)\subset\left\{(\rho,u,M,Q)\,:\,Q=\frac{2}{d}\left\langle\nu^{\eps}_{(t,x)},e_{\eps}\right\rangle \right\}
 		\end{align*}
 		for a.e.~$(t,x)$.
 		\item The function $(t,x)\mapsto\left\langle\nu^{\eps}_{(t,x)},e_{\eps}\right\rangle$ is continuous for all $\eps>0$.
 		\item The initial data $(\rho_0^{\eps},u_0^{\eps})$ converge strongly in $L^{\gamma}\times L^2$.
 	\end{itemize}
	Then $\mu$ is an energy admissible augmented measure-valued solution and the low Mach sequence $(\rho_k^{\eps_k},u_k^{\eps_k})$ is also energy admissible.
\end{prop}
\begin{rem}\label{rem:monoatomic}
	In Proposition~\ref{prop:characterizationapplication}, we considered only the case of monoatomic gas dynamics for the compressible model. This stems from the usage of the results of \cite{GW21} in the proof below. Here, the specific choice $\gamma=1+\frac{2}{d}$ yields a significant simplification due the fact that the energy density $e_{\eps}=\frac{1}{2}\rho|u|^2+\frac{1}{1-\gamma}\frac{1}{\eps}\rho^{\gamma}$ and the generalized pressure $Q=\frac{1}{\eps}\rho^{\gamma}+\rho\frac{|u|^2}{d}$ of lifted states $(\rho,u,M,Q)=\Theta^{\eps}(\rho,u)$ coincide up to multiplication with a constant. The most important reason for the lack of results beyond this special case is that the method of genuinely compressible convex integration from \cite{M21} and \cite{DSW21}, a crucial tool in \cite{GW21}, is so far only available for $\gamma=1+\frac{2}{d}$.
\end{rem}
\begin{proof}[Proof of Proposition~\ref{prop:characterizationapplication}.]
	Consider only $0<\eps<1$ small enough such that $\bar{\rho}-2\eps>0$. We apply Theorem~1.1 in \cite{GW21} on every $\nu^{\eps}$ to obtain a sequence $(\rho_j^{\eps},u_j^{\eps})$ of weak solutions of (\ref{eq:isentropiceulereps}) generating $\nu^{\eps}$. Due to the assumptions on the support of $\nu^{\eps}$ the sequence $(\rho_j^{\eps})$ can be chosen such that
	\begin{align*}
		\bar{\rho}-2\eps\leq \rho_j^{\eps} &\leq \bar{\rho}+2\eps.
	\end{align*}
	In particular, it holds that
	\begin{align*}
		\frac{1}{\eps}|(\rho_j^{\eps})^{\gamma}-\bar{\rho}^{\gamma}|\leq \frac{1}{\eps}\gamma(\bar{\rho}+2\eps)^{\gamma-1}|\rho_j^{\eps}-\bar{\rho}|\leq 2\gamma(\bar{\rho}+2)^{\gamma-1}.
	\end{align*}
	This implies that $\frac{1}{\eps}\left((\rho_j^{\eps})^{\gamma}-\bar{\rho}^{\gamma}\right)$ and $\rho_j^{\eps}$ are uniformly bounded. Moreover, $\rho_j^{\eps}$ is uniformly bounded away from zero for small enough $\eps$.\\
	For the boundedness of the sequence $(u_j^{\eps})$ observe that due to the usage of the convex integration method from \cite{DSW21} in the proof of Theorem~1.1 in \cite{GW21}, we obtain that
	\begin{align*}
		\rho_j^{\eps}\frac{|u_j^{\eps}|^2}{d}+\frac{1}{\eps}(\rho_j^{\eps})^{\gamma}=Q_j^{\eps}.
	\end{align*}
	Here, $Q_j^{\eps}$ is the generalized pressure corresponding to $(\rho_j^{\eps},u_j^{\eps})$ satisfying the estimate
	\begin{align*}
		Q_j^{\eps}\leq (\bar{\rho}+\eps)\frac{\chi}{d}+\frac{1}{\eps}(\bar{\rho}+\eps)^{\gamma}+1.
	\end{align*}
	The latter estimate can be guaranteed by the truncation step in the proof of Theorem~1.1 in \cite{GW21}, because on the support of $(\Theta^{\eps})_{\sharp}\nu^{\eps}$ we have that $Q\leq (\bar{\rho}+\eps)\frac{\chi}{d}+\frac{1}{\eps}(\bar{\rho}+\eps)^{\gamma}$. This is where the continuity of $\chi$ is used.\\ Hence, we obtain
	\begin{align*}
		(\bar{\rho}-2\eps)\frac{|u_j^{\eps}|^2}{d}+\frac{1}{\eps}(\bar{\rho}-2\eps)^{\gamma}\leq \rho_j^{\eps}\frac{|u_j^{\eps}|^2}{d}+\frac{1}{\eps}(\rho_j^{\eps})^{\gamma}\leq (\bar{\rho}+\eps)\frac{\chi}{d}+\frac{1}{\eps}(\bar{\rho}+\eps)^{\gamma}+1.
	\end{align*}
	From this we can infer that $(u_j^{\eps})$ is uniformly bounded in $j,\eps$ since the terms of order $\frac{1}{\eps}$ cancel in the limit. The assertion now follows from Corollary~\ref{cor:sequenceimpliessolution}.\\
	\\
	In the case of energy admissibility the same arguments as above can be applied to Theorem~1.2 in \cite{GW21} and its proof. This yields for every $\eps$ small enough a sequence $(\rho_j^{\eps},u_j^{\eps})$ of energy admissible weak solutions with initial data converging strongly to $(\rho_0^{\eps},u_0^{\eps})$ in $L^{\gamma}\times L^2$.
\end{proof}
\section{Necessary Jensen-type conditions for incompressible singular limits}\label{sect:necessaryjensentypeconditionsforincompressiblesingularlimits}

After investigating situations in which a low Mach limit exists, we now want to study necessary conditions for such a limit process. Actually, we will consider also other singular limits like vanishing viscosity limits and limits of weak solutions. The Jensen-type condition (\ref{eq:jensencondition}) plays a fundamental role in this context.\\
Let us start with an $L^2$-condition for vanishing viscosity limits of uniformly bounded, hence smooth, energy admissible solutions of the Navier-Stokes system.
\begin{theo}
	Let $(u_n)$ be a uniformly bounded, energy admissible vanishing viscosity sequence of incompressible Navier-Stokes solutions generating the Young measure $\mu$ with any initial data. Then $\mu$ is an augmented measure-valued solution of (\ref{eq:incompressibleeuler}) satisfying
	\begin{align*}
		\left\langle \left(S_{\#}\mu\right)_{(t,x)},f\right\rangle\geq Q_{\mathcal{A}_E}f\left(\left\langle \left(S_{\#}\mu\right)_{(t,x)},\operatorname{id}\right\rangle\right)
	\end{align*}
	for all $f\in C(\R^N)$ with $|f(\bullet)|\leq C(1+|\bullet|^2)$ and all $(t,x)\in \Omega_0$, where $\Omega_0\subset (0,T)\times \T^d$ is a set of full measure.
\end{theo}
\begin{proof}
	Since $u_n$ is uniformly bounded and weakly continuous in time, also its initial data $u_n^0$ is uniformly bounded. Hence, the sequence of initial data admits a weakly convergent subsequence $u_{n_k}^0\overset{L^2}{\rightharpoonup} u_0$. Moreover, from the energy inequality we obtain that
	\begin{align*}
		\|\mathcal{A}_E(S(u_n,P_n))\|_{W^{-1,2}\left((0,T)\times \T^d\right)}^2&=\underset{\underset{\varphi\in C_c^{\infty}\left((0,T)\times \T^d\right)}{\|\varphi\|_{W^{1,2}\left((0,T)\times \T^d\right)}\leq 1,}}{\sup}\left|\int\limits_{(0,T)\times \T^d}^{} \alpha_n\nabla u_n:\nabla\varphi\,\dx\dt\right|^2\\
		&\leq \alpha_n^2\|\nabla u_n\|^2_{L^2\left((0,T)\times \T^d\right)}\\
		&\leq \alpha_n\int\limits_{\T^d}^{}|u_n^0|^2\dx\\
		&\rightarrow 0,
	\end{align*}
	where $\alpha_n\rightarrow 0$ is the sequence of viscosity parameters and $P_n$ the corresponding pressure for $u_n$. It is now straightforward to check that the limit $\mu$ is an augmented measure-valued solution with initial data $u_0$.\\
	The uniform boundedness of $(u_n)$ guarantees a uniform $L^2$-bound on $P_n=(-\Delta)^{-1}\operatorname{div}\operatorname{div}(u_n\otimes u_n)$ by a Calder\'on-Zygmund argument. Hence, also the sequence $S(u_n,P_n)$ is uniformly bounded in $L^2$ and generates the lift $S_{\#}\mu$. Since $\|\mathcal{A}_E(S(u_n,P_n))\|_{W^{-1,2}}$ goes to zero as $n\rightarrow\infty$, we obtain from Lemma~2.15 in \cite{FM99} a $2$-equiintegrable sequence $(z_n)\subset L^2\left((0,T)\times \T^d\right)\cap \ker(\mathcal{A}_E)$ also generating $S_{\#}\mu$. Theorem~4.1 in \cite{FM99} now yields that for all $f\in C(\R^N)$ with $|f(\bullet)|\leq C(1+|\bullet|^2)$
	\begin{align*}
		\left\langle \left(S_{\#}\mu\right)_{(t,x)},f\right\rangle\geq Q_{\mathcal{A}_E}f\left(\left\langle \left(S_{\#}\mu\right)_{(t,x)},\operatorname{id}\right\rangle\right)
	\end{align*}
	holds on a set of full measure $\Omega_0\subset(0,T)\times \T^d$.	
\end{proof}

For uniformly bounded low Mach limits or limits of weak solutions generating $\mu$ the quadratic growth condition on $f$ can actually be dropped. This can be shown by using the following more abstract result.
\begin{prop}\label{prop:generalizedjensencond}
	Let $\mathcal{A}$ be a linear homogeneous differential operator with coefficient matrices acting on $\R^m$. Assume there exists an $\mathcal{A}$-free sequence $(z_n)$ which is uniformly bounded in $L^{\infty}\left((0,T)\times \T^d\right)$ by $R>0$ and generates a Young measure $\nu$. Then $\nu$ satisfies the Jensen condition
	\begin{align}
		\left\langle\nu_{(t,x)},f\right\rangle\geq Q_{\mathcal{A}}^{8R}f\left(\left\langle\nu_{(t,x)},\operatorname{id}\right\rangle\right)\label{eq:generalizedjenseninequality}
	\end{align}
	for all $f\in C(\R^N)$ and all $(t,x)\in\Omega_0$, where $\Omega_0\subset (0,T)\times \T^d$ is a set of full measure.
\end{prop}
\begin{proof}
	By Proposition~3.9 in \cite{FM99} and its proof there exists a set $\Omega_0\subset (0,T)\times \T^d$ of full measure such that for every $(\tau,y)\in\Omega_0$ and every subcube $\mathcal{Q}'\subset\subset (0,1)^{d+1}$ there exists an $\mathcal{A}$-free sequence $(\bar{z}_n)\in L^{\infty}(\T^{d+1})$ with $\|\bar{z}_n\|_{L^{\infty}}\leq 4R$ and $\underset{\T^{d+1}}{\dashint}\bar{z}_n=\left\langle\nu_{(\tau,y)},\operatorname{id}\right\rangle$, and such that $(\bar{z}_n)$ generates a Young measure $\mu$ satisfying
	\begin{align*}
		\left|\int\limits_{\T^{d+1}}^{}\left\langle\mu_{(t,x)},f\right\rangle\dt\dx-\left\langle\nu_{(\tau,y)},f\right\rangle \right|\leq \|f\|_{L^{\infty}\left(B_{3R}(0) \right)}\left|(0,1)^{d+1}\backslash \mathcal{Q}' \right|
	\end{align*}
	for all $f\in C_0(\R^m)$.\\
	Fix $(\tau,y)\in\Omega_0$. Let $f\in C_0(\R^m)$ and let $\eps>0$. Choose a subcube $\mathcal{Q}'\subset \subset (0,1)^{d+1}$ such that we have $\|f\|_{L^{\infty}}\left|(0,1)^{d+1}\backslash \mathcal{Q}'\right|\leq \eps$. Then the corresponding sequence $(\bar{z}_n)$ satisfies
	\begin{align*}
		\eps+\left\langle \nu_{(\tau,y)},f\right\rangle\geq \int\limits_{\T^{d+1}}^{}\left\langle\mu_{(t,x)},f\right\rangle\dt\dx=\lim\limits_{n\rightarrow\infty}\int\limits_{\T^{d+1}}^{}f(\bar{z}_n(t,x))\dt\dx.
	\end{align*}
	If $\eta_{\delta}$ denotes a standard mollifier, then $\eta_{\delta}\ast\bar{z}_n\rightarrow \bar{z}_n$ in $L^1(\T^{d+1})$. So, we can extract a subsequence such that $\eta_{\delta_n}\ast\bar{z}_n\overset{Y}{\rightharpoonup}\mu$. Observe that $\eta_{\delta_n}\ast\bar{z}_n\in C^{\infty}(\T^{d+1})$ is $\mathcal{A}$-free, has the same average as $\bar{z}_n$, and admits the same uniform bound $4R$. Therefore, by the definition of the truncated $\mathcal{A}$-quasiconvex envelope we have
	\begin{align*}
		\lim\limits_{n\rightarrow\infty}\int\limits_{\T^{d+1}}^{}f(\bar{z}_n(t,x))\dt\dx&=\lim\limits_{n\rightarrow\infty}\int\limits_{\T^{d+1}}^{}f((\eta_{\delta_n}\ast\bar{z}_n)(t,x))\dt\dx\\
		&\geq \underset{n\rightarrow\infty}{\limsup}\,Q_{\mathcal{A}}^{8R}f\left(\int\limits_{\T^{d+1}}^{}(\eta_{\delta_n}\ast\bar{z}_n)(t,x)\dt\dx \right)\\
		&=Q_{\mathcal{A}}^{8R}f\left(\left\langle\nu_{(\tau,y)},\operatorname{id}\right\rangle \right),
	\end{align*}
	where we used that the $\mathcal{A}$-free function $\eta_{\delta_n}\ast\bar{z}_n-\int\limits_{\T^{d+1}}^{}(\eta_{\delta_n}\ast\bar{z}_n)(t,x)\dt\dx$ is average-free and bounded by $8R$. Moreover, we used the continuity of the truncated quasiconvex envelope, which is not hard to check, cf. the proof of Proposition~3.4 in \cite{FM99}. As $\eps>0$ and $f$ were arbitrary, putting everything together yields
	\begin{align*}
		\left\langle\nu_{(\tau,y)},f\right\rangle\geq Q_{\mathcal{A}}^{8R}f\left(\left\langle\nu_{(\tau,y)},\operatorname{id}\right\rangle\right)
	\end{align*}
	for all $f\in C_0(\R^m)$. Note that $\supp\nu\subset B_R(0)$ almost everywhere by the boundedness of its generating sequence $(z_n)$. Furthermore, for two continuous functions $f$ and $g$ with $f=g$ on $B_{9R}(0)$ it is immediate from the definition of the truncated envelope that
	\begin{align*}
		Q_{\mathcal{A}}^{8R}f\left(\left\langle\nu_{(\tau,y)},\operatorname{id}\right\rangle \right)=Q_{\mathcal{A}}^{8R}g\left(\left\langle\nu_{(\tau,y)},\operatorname{id}\right\rangle \right).
	\end{align*}
	Since we also have $\left\langle\nu_{(\tau,y)},f\right\rangle=\left\langle\nu_{(\tau,y)},g\right\rangle$, the Jensen inequality (\ref{eq:generalizedjenseninequality}) holds for all $f\in C(\R^m)$.	
\end{proof}

With this at hand we are able to prove our main result Theorem~\ref{theo:jensencond}.
\begin{proof}[Proof of Theorem \ref{theo:jensencond}.]
	First consider the case of a low Mach sequence $(\rho^{\eps_n},u^{\eps_n})$. The functions $S\left(P_{\bar{\rho}}^{\eps_n}(\rho^{\eps_n},u^{\eps_n})\right)$ generate the lift $S_{\#}\mu$. Moreover, the sequence $\frac{1}{\eps_n}\left|(\rho^{\eps_n})^{\gamma}-\bar{\rho}^{\gamma}\right|$ is uniformly bounded. Thus, since $\eps_n\rightarrow 0$, we necessarily obtain $\rho^{\eps_n}\rightarrow \bar{\rho}$ almost everywhere on $(0,T)\times \T^d$. The uniform boundedness of $u^{\eps_n}$ then yields that
	\begin{align*}
		S\left(u^{\eps_n},\frac{1}{\eps_n\bar{\rho}}\left((\rho^{\eps_n})^{\gamma}-\bar{\rho}^{\gamma}\right)\right)-\frac{1}{\bar{\rho}}C^{\eps_n}_{\bar{\rho}}(\rho^{\eps_n},u^{\eps_n})\rightarrow 0
	\end{align*}
	almost everywhere. By standard Young measure theory we then obtain that $\frac{1}{\bar{\rho}}C^{\eps_n}_{\bar{\rho}}(\rho^{\eps_n},u^{\eps_n})\overset{Y}{\rightharpoonup}S_{\#}\mu$. Note that the sequence of lifts $\left(\frac{1}{\bar{\rho}}C^{\eps_n}_{\bar{\rho}}(\rho^{\eps_n},u^{\eps_n})\right)_{n\in\N}$ is $\mathcal{A}_E$-free and uniformly bounded, say by $R>0$.\\
	The case of weak incompressible solutions $u_n$ with corresponding uniformly bounded pressure $P_n$ is even easier, since then directly $S(u_n,P_n)$ is $\mathcal{A}_E$-free, generates $S_{\#}\mu$, and is uniformly bounded, say by $R>0$.\\
	In both cases we can apply Proposition~\ref{prop:generalizedjensencond} to get
	\begin{align*}
		\left\langle (S_{\#}\mu)_{(t,x)},f\right\rangle\geq Q_{\mathcal{A}_E}^{8R}f\left(\left\langle (S_{\#}\mu)_{(t,x)},\operatorname{id}\right\rangle \right)\geq Q_{\mathcal{A}_E}f\left(\left\langle (S_{\#}\mu)_{(t,x)},\operatorname{id}\right\rangle \right)
	\end{align*}
	for all $f\in C(\R^N)$ and all $(t,x)\in\Omega_0$, where $\Omega_0\subset (0,T)\times \T^d$ is a set of full measure, which finishes the proof.
\end{proof}
\begin{rem}
	We observed in Proposition~\ref{prop:characterizationapplication} that a Jensen-type condition on a low Mach sequence of Young measures $\nu^{\eps}$ generating $\mu$ is part of sufficient conditions for $\mu$ to be a low Mach limit of weak solutions.
	On the other hand, we derived a weaker Jensen condition as a necessary condition.\\
	It would be interesting to obtain a sufficient set of conditions including a Jensen condition only for $\mu$ itself and not for an already given low Mach sequence $\nu^{\eps}$. This should be compared to the problem of generating measure-valued solutions of the isentropic Euler system by weak solutions. Here, the same relaxed system (\ref{eq:relaxedeulersystem}) is used, but sufficient conditions for the limit measure itself can be found, cf. \cite{GW21}.\\
	Note also that for the two envelopes $Q_{\mathcal{A}_E}^qf$ and $Q_{\mathcal{B}_E}^qf$ of a continuous function $f$ one can only show that $Q_{\mathcal{A}_E}^qf\leq Q_{\mathcal{B}_E}^qf$. But for the converse inequality a Calder\'on-Zygmund estimate from $L^{\infty}$ to $L^{\infty}$ would be necessary, which in general does not hold. So, already conceptional gaps between sufficient and necessary conditions have to be bridged for a complete description of Young measures being low Mach limits.
\end{rem}

\section{The pressure-free case}\label{sect:thepressurefreecase}
For the special class of pressure-free solutions we can in fact already give sufficient conditions in the spirit of Proposition~\ref{prop:characterizationapplication} but only on the incompressible solution itself in order to infer the existence of a generating low Mach sequence of weak solutions. The reason for the simplicity of this situation lies in the observation that a pressure-free solution of (\ref{eq:incompressibleeuler}) is also a solution of (\ref{eq:isentropiceulereps}) for any $\eps>0$.\\
As we use the results from \cite{GW21} again, we constrain ourselves to the case of monoatomic gases, i.e.~$\gamma=1+\frac{2}{d}$ in this section, see also Remark~\ref{rem:monoatomic}.
\begin{prop}\label{prop:pressurefreesituation}
	Let $d\geq 2$ and $\gamma=1+\frac{2}{d}$. Assume $\nu$ is a Young measure with the following properties:
	\begin{itemize}
		\item The measure $\nu$ is an energy admissible augmented measure-valued solution of (\ref{eq:incompressibleeuler}) with initial data $u_0\in L^1(\T^d)$ such that $\operatorname{div}u_0=0$.
		\item There exists some $R>0$ such that
		\begin{align}
			\left\langle \left(S_\#\nu\right)_{(t,x)},f\right\rangle\geq Q^R_{\mathcal{B}_E}f\left(\left\langle \left(S_\#\nu\right)_{(t,x)},\operatorname{id}\right\rangle\right)\label{eq:pressurefreejensen}
		\end{align}
		for all $f\in C(\R^N)$ and all $(t,x)\in\Omega_0$, where $\Omega_0\subset (0,T)\times \T^d$ is a set of full measure.
		\item There exist $\sigma\in C\left([0,T]\times \T^d\right)$ and $w\in W^{2,\infty}\left((0,T)\times \T^d\right)$ such that
		\begin{align*}
			\left\langle S_\#\nu,\operatorname{id}\right\rangle=\sigma+\mathcal{B}_Ew.
		\end{align*}
		\item There exists a bounded and continuous non-negative function $\chi$ of time such that
		\begin{align*}
			\supp(\nu_{(t,x)})\subset \{(u,P)\,:\,|u|^2=\chi(t),\ P=0 \} 
		\end{align*}
		for a.e.~$(t,x)\in (0,T)\times \T^d$.
	\end{itemize}
	Then we obtain a low Mach sequence $(\rho^{\eps_k},u^{\eps_k})$ of energy admissible weak solutions such that the sequence $P_{1}^{\eps_k}(\rho^{\eps_k},u^{\eps_k})$ generates $\nu$.
\end{prop}
\begin{proof}
	Define the measures
	\begin{align*}
		\nu^{\eps}:=\delta_1\otimes \operatorname{pr}_u(\nu)
	\end{align*}
	for all $\eps>0$. We check all prerequisites of Proposition~\ref{prop:characterizationapplication} and apply this to the sequence $(\nu^{\eps})_{\eps>0}$ which then yields the assertion.\\
	First, notice that it is straightforward to check that $\nu^{\eps}$ solves (\ref{eq:isentropiceulereps}) with initial data $(1,u_0)$, since the density is constant and equal to one. The convergence of the initial data is therefore trivial. Now let $f\in C(\R^N)$ and observe that
	\begin{align*}
		\left\langle (\Theta^{\eps})_\#\nu^{\eps},f\right\rangle&=\left\langle\delta_1\otimes\operatorname{pr}_u(\nu),f\left(\rho,\rho u,\rho u\ocircle u,\frac{\rho|u|^2}{d}+\frac{1}{\eps}\rho^{\gamma} \right) \right\rangle\\
		&=\left\langle\operatorname{pr}_u(\nu)\otimes \delta_0,f\left(1,u,u\ocircle u,P+\frac{|u|^2}{d}+\frac{1}{\eps} \right) \right\rangle\\
		&=\left\langle S_\#\nu,f\left((\rho,u,M,Q)+\left(0,0,0,\frac{1}{\eps} \right) \right)\right\rangle.
	\end{align*}
	Here we used that $\nu=\operatorname{pr}_u(\nu)\otimes\delta_0$ holds due to the support properties of $\nu$. 
	Therefore, by (\ref{eq:pressurefreejensen}) for $(t,x)\in \Omega_0$ it holds that
	\begin{align*}
		\left\langle \left((\Theta^{\eps})_\#\nu^{\eps}\right)_{(t,x)},f\right\rangle&=\left\langle \left(S_\#\nu\right)_{(t,x)},f\left(\operatorname{id}+\left(0,0,0,\frac{1}{\eps} \right) \right)\right\rangle\\
		&\geq Q^R_{\mathcal{B}_E}f\left(\left\langle\left(S_\#\nu\right)_{(t,x)},\operatorname{id}\right\rangle+\left(0,0,0,\frac{1}{\eps} \right)\right)\\
		&= Q^R_{\mathcal{B}_E}f\left(\left\langle\left((\Theta^{\eps})_\#\nu^{\eps}\right)_{(t,x)},\operatorname{id}\right\rangle\right),
	\end{align*}
	where we used the fact that $f\left(\bullet+\left(0,0,0,\frac{1}{\eps} \right) \right)$ is again a continuous function on $C(\R^N)$. This yields the validity of (\ref{eq:sufficientjensencond}).\\
	Next, observe that
	\begin{align*}
		\left\langle(\Theta^{\eps})_\#\nu^{\eps},\operatorname{id}\right\rangle=\left\langle S_\#\nu,\operatorname{id}\right\rangle+\left(0,0,0,\frac{1}{\eps} \right)=\sigma^{\eps}+\mathcal{B}_Ew.
	\end{align*}
	Here, we defined $\sigma^{\eps}:=\sigma+\left(0,0,0,\frac{1}{\eps} \right)\in C\left([0,T]\times \T^d\right)$.\\
	The support properties of $\nu^{\eps}$ are also fitting into the setting of Proposition~\ref{prop:characterizationapplication} since on the support of $(\Theta^{\eps})_\#\nu^{\eps}$ it holds that $\rho=1$, $|u|^2=\chi$, and $Q=\frac{\rho|u|^2}{d}+\frac{1}{\eps}\rho^{\gamma}=\frac{1}{d}\chi+\frac{1}{\eps}=\frac{2}{d}\left\langle \nu^{\eps},e_{\eps}\right\rangle$. It is clear that for all $\eps>0$ the function $\left\langle\nu^{\eps},e_{\eps}\right\rangle=\frac{1}{2}\chi+\frac{d}{2}\frac{1}{\eps}$ is continuous. It only remains to prove that $\nu^{\eps}$ is energy admissible. But this is straightforward to check as $\nu$ is energy admissible.
\end{proof}
We can find examples in terms of di-atomic measures for the above situation if we assume the existence of certain weak solutions, cf. Remark \ref{rem:chiodaroliconjecture} below.
\begin{cor}\label{cor:pressurefreeexample}
	Let $d=2$ and $\gamma=2$. Let $u_0\in L^1(\T^d)$ with $\operatorname{div}u_0=0$ such that there exist uniformly bounded energy admissible weak solutions $u_1\neq u_2$ of (\ref{eq:incompressibleeuler}) with pressure identically zero and initial data $u_0$. Moreover, suppose that there exists some bounded and continuous non-negative function $\chi$ of time such that $|u_1|^2=\chi=|u_2|^2$ almost everywhere. Let further $\lambda\in (0,1)$ be such that $\lambda S(u_1,0)+(1-\lambda)S(u_2,0)=\sigma+\mathcal{B}_Ew$ for some $\sigma\in C\left([0,T]\times \T^d\right)$ and $w\in W^{2,\infty}\left((0,T)\times \T^d\right)$.\\
	Then there exists a low Mach sequence $(\rho^{\eps_k},u^{\eps_k})$ of energy admissible weak solutions such that the pressure lifts $P_{1}^{\eps_k}(\rho^{\eps_k},u^{\eps_k})$ generate the augmented measure-valued solution $\nu:=\lambda\delta_{(u_1,0)}+(1-\lambda)\delta_{(u_2,0)}$ of (\ref{eq:incompressibleeuler}).
\end{cor}
\begin{proof}
	It is straightforward to check that $\nu$ is an energy admissible augmented measure-valued solution with initial data $u_0$. Note that the densities of the states $z_1:=S(u_1,0)$ and $z_2:=S(u_2,0)$ are both equal to one. Thus, by Remark 4.10 in \cite{GW20} we obtain that $(z_1-z_2)\in \underset{|\xi|=1}{\bigcup}\ker \mathbb{A}_E(\xi)$ almost everywhere. By Lemma~5.6 in \cite{GW21} this implies that $(z_1-z_2)\in \underset{|\xi|=1}{\bigcup}\image \mathbb{B}_E(\xi)$ almost everywhere. Now we apply Lemma~3.2 in \cite{GW21} to infer that
	\begin{align*}
		\left\langle S_\#\nu,f\right\rangle=\lambda f(z_1)+(1-\lambda)f(z_2)\geq Q_{\mathcal{B}_E}^{C\|z_1-z_2\|_{L^{\infty}}}f(\lambda z_1+(1-\lambda)z_2)=Q_{\mathcal{B}_E}^{C\|z_1-z_2\|_{L^{\infty}}}f\left(\left\langle S_\#\nu,\operatorname{id}\right\rangle \right)
	\end{align*}
	for all $f\in C(\R^N)$ on a subset of $(0,T)\times \T^d$ of full measure. Since both $z_1$ and $z_2$ are uniformly bounded, this implies (\ref{eq:pressurefreejensen}).\\
	By assumption it holds that $\langle S_\#\nu,\operatorname{id}\rangle=\sigma+\mathcal{B}_Ew$ and for all $(u,P)$ in the support of $\nu$ we have $P=0$ and $|u|^2=|u_i|^2=\chi$ for some $i\in\{1,2\}$. Therefore, we checked all assumptions of Proposition~\ref{prop:pressurefreesituation} and the assertion follows.
\end{proof}

\begin{rem}\label{rem:chiodaroliconjecture}
	The method of convex integration gives us the following:\\
	By choosing the pressure function $\rho\mapsto P(\rho)$ arbitrarily and $\rho_0\equiv 1$ in Theorem~2.1 in \cite{C14} there exists some initial data $u_0\in L^{\infty}(\T^2)$ giving rise to infinitely many weak solutions $(1,u)$ of (\ref{eq:isentropiceulereps}) with constant density and the property that $|u|^2=\chi$ and $|u_0|^2=\chi(0)$ for some fixed smooth and bounded function of time $\chi$. Hence, the functions $u$ are weak solutions of (\ref{eq:isentropiceulereps}) with constant pressure. So, they solve the equations also for pressure identically zero. 
	In particular, by Theorem~2.2 in \cite{C14} the function $\chi$ can be chosen such that the solutions $u$ are energy admissible. Since the density is constant, we necessarily have that $u_0$ is divergence-free.\\
	This ensures the existence of (infinitely many) examples for applying Corollary~\ref{cor:pressurefreeexample} if we can also guarantee that some convex combination of two such weak solutions can be written as $\sigma+\mathcal{B}_Ew$ for $\sigma$ continuous and $w\in W^{2,\infty}$. Unfortunately, there is no rigorous proof of this statement yet, but we conjecture that this should indeed be the case if one applies the convex integration scheme carefully enough, cf. Remark 5.5 in \cite{GW21}.
\end{rem}

\section{Examples violating the Jensen condition}\label{sect:examplesviolatingthejensencondition}
We now describe a class of measure-valued solutions violating the Jensen condition (\ref{eq:jensencondition}). In particular this condition is not trivially fulfilled if we can find at least one solution in this class. Such examples will be given in Corollary~\ref{cor:example}, below.
\begin{prop}\label{prop:notwaveconeconnexample}
	Let $d=2$ and $\lambda\in (0,1)$. Let $u_1,u_2$ be weak solutions of (\ref{eq:incompressibleeuler}) for the same initial data with corresponding pressures $P_1$ and $P_2$, respectively. Then the augmented measure-valued solution $\nu:=\lambda \delta_{(u_1,P_1)}+(1-\lambda)\delta_{(u_2,P_2)}$ does not fulfill the Jensen condition (\ref{eq:jensencondition}) for all $f\in C(\R^N)$ if and only if $u_1\neq u_2$ and $P_1\neq P_2$ on a set with positive Lebesgue measure.
\end{prop}
\begin{rem}
	Actually, from the proof of Proposition~\ref{prop:notwaveconeconnexample} we obtain that a di-atomic measure satisfies the Jensen condition (\ref{eq:jensencondition}) for all $f\in C(\R^N)$ if and only if it satisfies (\ref{eq:jensencondition}) for all $f\in C(\R^N)$ with linear growth. In particular, di-atomic measures consisting of weak solutions with non-wave-cone-connected lifts cannot be generated by a low Mach number limit, a sequence of weak solutions, or an energy admissible vanishing viscosity sequence, cf. (\ref{eq:waveconeconnectedcalculation}) below.
\end{rem}
\begin{proof}[Proof of Proposition \ref{prop:notwaveconeconnexample}]
	An elementary computation shows that the incompressible lifts $z_i:=S(u_i,P_i)$ are $\mathcal{A}_E$-wave-cone-connected if and only if
	\begin{align}
		0=\det\left(\left(\sum\limits_{j=1}^{N}A_{ij}^l\left(z_1^j-z_2^j\right) \right)_{jl}\right)=-|u_1-u_2|^2\cdot(P_1-P_2)\label{eq:waveconeconnectedcalculation},
	\end{align}
	where $A^l$ are the coefficient matrices of $\mathcal{A}_E$, see e.g. \cite{GW20}. Thus, it suffices to show that the states $z_1$ and $z_2$ are wave-cone-connected almost everywhere if and only if the Jensen condition (\ref{eq:jensencondition}) holds.\\
	Assume first that the states $z_1$ and $z_2$ are not wave-cone-connected on a set with positive measure. Then by Theorem~3.5 in \cite{GW20} the Young measure $S_{\#}\nu=\lambda\delta_{S(u_1,P_1)}+(1-\lambda)\delta_{S(u_2,P_2)}$ cannot be generated by $\mathcal{A}_E$-free equiintegrable sequences. Hence, as $S_{\#}\nu$ has compact support and an $\mathcal{A}_E$-free barycenter, by Theorem~4.1 in \cite{FM99} the Jensen condition
	\begin{align*}
		\langle S_{\#}\nu,f\rangle\geq Q_{\mathcal{A}_E}f(\langle S_{\#}\nu,\operatorname{id}\rangle)
	\end{align*}
	must be violated for some $f\in C(\R^N)$ with linear growth on a set of positive measure. In particular, if the Jensen condition (\ref{eq:jensencondition}) holds almost everywhere and for all continuous functions, then $z_1$ and $z_2$ must be wave-cone-connected almost everywhere.\\
	Conversely, if $z_1$ and $z_2$ are wave-cone-connected almost everywhere, then we know for all $f\in C(\R^N)$ from Proposition~3.4 in \cite{FM99} that $Q_{\mathcal{A}_E}f$ is convex on the line $[z_1(t,x),z_2(t,x)]$ for almost every $(t,x)$. Therefore, by using the classical Jensen inequality we have
	\begin{align*}
		\left\langle (S_{\#}\nu)_{(t,x)},f\right\rangle\geq \left\langle (S_{\#}\nu)_{(t,x)},Q_{\mathcal{A}_E}f\right\rangle\geq Q_{\mathcal{A}_E}f\left(\left\langle (S_{\#}\nu)_{(t,x)},\operatorname{id}\right\rangle\right).
	\end{align*}
	This finishes the proof.
\end{proof}

We find examples for applying the above result by the method of convex integration.
\begin{cor}\label{cor:example}
	Let $d=2$. There exists some initial data $u_0\in L^{\infty}(\T^d)$ with $\operatorname{div}u_0=0$ which gives rise to an augmented measure-valued solution that does not fulfill the Jensen condition (\ref{eq:jensencondition}) on a set of positive measure.
\end{cor}
\begin{proof}
	Observe that $(0,0,0)$ is a continuous solution of the relaxed Euler system (\ref{eq:relaxedeulersystem}) with $\rho\equiv 0$. Define $\Omega:=\left(\frac{1}{2},\frac{3}{4} \right)^2$ and choose a function $\bar{e}\in C\left((0,T)\times \T^2\right)\cap C\left([0,T],L^1(\T^2)\right)$ such that $\bar{e}=0$ on $\T^2\backslash \Omega$ and $\bar{e}>0$ on $\Omega$ for all $t\in (0,T)$. Then Proposition~2 in \cite{DS10} implies the existence of infinitely many weak solutions $u\in L^{\infty}$ of (\ref{eq:incompressibleeuler}) on $[0,T)\times \T^2$ with corresponding pressure $P'=-\frac{1}{2}|u|^2=-\bar{e}\cdot\mathds{1}_{\Omega}$ a.e.~and $u=0$ on $\T^2\backslash\Omega$. It also holds that $u(0,x)=0$ for a.e.~$x\in\T^2$. As we are working on the torus $\T^2$ we choose the offset for the pressure such that it is average-free, i.e. we consider $P=P'-\underset{\T^2}{\dashint}P'\dx$.\\
	Thus, the initial data $u_0=0$ gives rise to the trivial solution $0$ and actually infinitely many weak solutions $(u,P)$. Moreover, as the continuous function $\bar{e}$ cannot have a constant nonzero value on $\Omega$, we know that $u\neq 0$ and $P\neq 0$ on a set of positive measure. Proposition~\ref{prop:notwaveconeconnexample} then yields the assertion.
\end{proof}

\begin{rem}
	Interestingly, in the situation of Corollary~\ref{cor:example} we observe the following:\\
	On the one hand, the augmented measure-valued solution $\nu=\lambda \delta_{(u_1,P_1)}+(1-\lambda)\delta_{(u_2,P_2)}$ cannot be generated by weak solutions of the incompressible Euler system. On the other hand, the measure-valued solution $\lambda\delta_{u_1}+(1-\lambda)\delta_{u_2}$ considered as a classical measure-valued solution, can always be generated by weak solutions $u_n\in L^{\infty}$ of (\ref{eq:incompressibleeuler}) by the results of \cite{SW12}. This seems to be a contradiction. However, it only tells us that if we write out the (dependent) pressure variable again, then $(u_n,P_n)$ generates some corresponding augmented measure-valued solution if $(u_n,P_n)$ is, say, uniformly $L^2$-bounded, but it cannot generate the measure $\nu$. This indicates that the solution concept of classical measure-valued solutions is too coarse to encode the information of low Mach limits. In particular, augmented measure-valued solutions are a proper generalization of the classical notion of measure-valued solution.
\end{rem}
As weak and measure-valued solutions of the incompressible Euler system are vastly non-unique, we should come up with some sort of selection principle.
\begin{rem}\label{rem:selectioncriterion}
	Since every natural fluid is compressible to some extent, the results we obtained in the present paper might suggest the following selection criterion for measure-valued solutions as formulated already in the introduction:\\
	\\
	\textit{We may discard every augmented measure-valued solution $\mu$ of (\ref{eq:incompressibleeuler}) as unphysical if it cannot be generated by a low Mach sequence of weak solutions of (\ref{eq:isentropiceulereps}).}\\
	\\
	This should be compared to, or even complemented by, the selection via vanishing viscosity limits, which is discussed in \cite{BTW12}. These two selection criteria might be actually connected, as we already saw that the Jensen condition (\ref{eq:jensencondition}) provides a necessary feature which both vanishing viscosity and low Mach limits have to fulfill. We also investigated the pressure-free case in which we are in fact able to decide for some solutions if they are a low Mach limit although they are not known to be a vanishing viscosity limit. Indeed, if $u_1$ and $u_2$ are weak solutions of the pressure-free Euler system and satisfy the prerequisites of Corollary~\ref{cor:pressurefreeexample}, then the measure $\lambda\delta_{(u_1,0)}+(1-\lambda)\delta_{(u_2,0)}$ is a low Mach limit.\\
	As a consequence one could justify that classical measure-valued solutions of (\ref{eq:incompressibleeuler}) should be discarded as unphysical if they have no extension as an augmented solution being a low Mach limit. However, this is quite speculative since it is not even clear the latter criterion is non-trivial. We only obtained necessary conditions for augmented solutions. It could still be possible that every classical measure-valued solution has an extension which is a low Mach limit. Further study on sharper necessary conditions will be needed for that.
\end{rem}
Let us conclude with the following observation.
\begin{rem}
	Consider a situation as in Theorem~\ref{theo:measuresequenceimpliessolution} with measure-valued solutions $\nu^{\eps}$ of (\ref{eq:isentropiceulereps}) and additionally assume that the initial data $(\rho^{\eps}_0,u^{\eps}_0)$ is \textit{well-prepared} in the sense of \cite{FKM19}, i.e.
	\begin{align*}
		\int\limits_{\T^d}^{}\frac{1}{2}\rho^{\eps}_0\left|u^{\eps}_0-u_0 \right|^2+\frac{1}{\eps}\frac{1}{\gamma-1}\left((\rho^{\eps}_0)^{\gamma}-\gamma\bar{\rho}^{\gamma-1}(\rho^{\eps}_0-\bar{\rho})-\bar{\rho}^{\gamma} \right)\dx\rightarrow 0\text{ as }\eps\rightarrow 0.
	\end{align*}
	If we also assume $u_0\in W^{k,2}\left(\T^d,\R^d\right)$ for some $k>\frac{d}{2}+1$, then the unique strong solution $U$ of (\ref{eq:incompressibleeuler}) with regularity $U\in C^1$ exists up to some time $T_{\operatorname{max}}$. Theorem~3.1 in \cite{FKM19} yields for $T<T_{\max}$ that
	\begin{align*}
		\underset{t\in(0,T)}{\esssup}\int\limits_{\T^d}^{}\left\langle(\nu^{\eps})_{(t,x)},\frac{1}{2}\rho|u-U(t,x)|^2+\frac{1}{\eps}\frac{1}{\gamma-1}\left(\rho^{\gamma}-\gamma\bar{\rho}^{\gamma-1}(\rho-\bar{\rho})-\bar{\rho}^{\gamma} \right)\right
		\rangle\dx\rightarrow 0
	\end{align*}
	as $\eps\rightarrow 0$. Since $\rho^{\gamma}-\gamma\bar{\rho}^{\gamma-1}(\rho-\bar{\rho})-\bar{\rho}^{\gamma}\geq 0$ for all $\rho>0$ this implies that
	\begin{align*}
		\left\langle(\nu^{\eps})_{(t,x)},\frac{1}{\eps}\frac{1}{\gamma-1}\left(\rho^{\gamma}-\gamma\bar{\rho}^{\gamma-1}(\rho-\bar{\rho})-\bar{\rho}^{\gamma} \right)\right
		\rangle\rightarrow 0\text{ in }L^1((0,T)\times \T^d).
	\end{align*}
	Hence, by using the uniformly compact support of $\nu^{\eps}$ this also implies weak* convergence in $L^{\infty}$. On the other hand, we know from Theorem~\ref{theo:measuresequenceimpliessolution} that $\left\langle\nu^{\eps},\frac{1}{\eps\bar{\rho}}(\rho^{\gamma}-\bar{\rho}^{\gamma})\right\rangle\overset{*}{\rightharpoonup}\langle\mu,P\rangle$ in $L^{\infty}$ for some augmented solution $\mu$.\\
	Therefore, we obtain that
	\begin{align*}
		\left\langle\nu^{\eps},\frac{1}{\eps\bar{\rho}}\gamma\bar{\rho}^{\gamma-1}(\rho-\bar{\rho})\right\rangle\overset{*}{\rightharpoonup}\langle\mu,P\rangle \text{ in }L^{\infty}.
	\end{align*}
	There is also an informal argument for this fact:\\
	Assume that the variables of (\ref{eq:isentropiceulereps}) can be expanded in terms of $\eps$, i.e. for example for the pressure it holds that $P=P^0+\eps P^1+\eps^2 P^2 +\ldots$ and similarly for $u$ and $\rho$. If the zeroth order term of the density fulfills $\rho^0=\bar{\rho}$, then collecting terms of the same order in $\eps$ and letting $\eps\rightarrow 0$ yields that $\left(u^0,\frac{P^1}{\rho^0}\right)$ is a solution of (\ref{eq:incompressibleeuler}). Moreover, by using the explicit form of $P(\rho)=\rho^{\gamma}$ we can Taylor expand it also as
	\begin{align*}
		P=\rho^{\gamma}&=(\rho(0))^{\gamma}+\eps \cdot\left(\frac{d}{d\eps}\bigg{|}_{\eps=0}(\rho(\eps)^{\gamma})\right)+\mathcal{O}(\eps^2)\\
		&=\bar{\rho}^{\gamma}+\eps\cdot \gamma\bar{\rho}^{\gamma-1}\rho^1+\mathcal{O}(\eps^2).
	\end{align*}
	Thus,
	\begin{align*}
		\lim\limits_{\eps\rightarrow 0}\frac{1}{\eps\bar{\rho}}\gamma\bar{\rho}^{\gamma-1}(\rho-\bar{\rho})=\gamma\bar{\rho}^{\gamma-2}\rho^1=\lim\limits_{\eps\rightarrow 0}\frac{1}{\eps\bar{\rho}}(\rho^{\gamma}-\bar{\rho}^{\gamma})=\lim\limits_{\eps\rightarrow 0}\frac{1}{\eps\rho^0}(P-P^0)= \frac{P^1}{\rho^0}
	\end{align*}
	and $\frac{P^1}{\rho^0}$ corresponds to the barycenter $\langle\mu,P\rangle$.
\end{rem}

\section*{Acknowledgements}
The author would like to thank Emil Wiedemann for introducing him to the problem of low Mach number limits and for many insightful discussions. The author is also very grateful for the suggestions and comments of the unknown referees.


\begin{thebibliography}{99}
	
	\bibitem{BTW12} C.~Bardos, E.~S.~Titi, E.~Wiedemann, The vanishing viscosity as a selection principle for the Euler equations: The case of 3D shear flow. {\em C.~R.~Math.~Acad.~Sci.~Paris}~{\bf 350}~(2012), no.~15--16, 757--760.
	
	\bibitem{BDS11} Y.~Brenier, C.~De~Lellis, L.~Sz\'ekelyhidi~Jr., Weak-strong uniqueness for measure-valued solutions. {\em Comm.~Math.~Phys.}~{\bf 305}~(2011), no.~2, 351--361.
	
	\bibitem{C14} E.~Chiodaroli, A counterexample to well-posedness of entropy solutions to the compressible Euler system. {\em J.~Hyperbolic~Differ.~Equ.}~{\bf 11}~(2014), no.~3, 493--519.
	
	\bibitem{CFKW15} E.~Chiodaroli, E.~Feireisl, O.~Kreml, E.~Wiedemann, $\mathcal{A}$-free rigidity and applications to the compressible Euler system. {\em Ann.~Math.~Pura~Appl.}~{\bf 196}~(2015), no.~4, 1557--1572.
	
	\bibitem{DSW21} T.~D\k{e}biec, J.~W.~D.~Skipper, E.~Wiedemann, A general convex integration scheme for the isentropic compressible Euler equations. {\em arXiv preprint}~\href{https://arxiv.org/abs/2107.10618}{arXiv:2107.10618}~(2021).
	
	\bibitem{DS09} C.~De~Lellis, L.~Székelyhidi~Jr., The Euler equations as a differential inclusion. {\em Ann.~Math.~(2)}~{\bf 170}~(2009), no.~3, 1417--1436.
	
	\bibitem{DS10} C.~De~Lellis, L.~Sz\'ekelyhidi~Jr., On admissibility criteria for weak solutions of the Euler equations. {\em Arch.~Ration.~Mech.~Anal.}~{\bf 195}~(2010), no.~1, 225--260.	
	
	\bibitem{D85} R.~J.~DiPerna, Measure-valued solutions to conservation laws. {\em Arch.~Rational~Mech.~Anal.}~{\bf 88}~(1985), no.~3, 223--270.
	
	\bibitem{DM87} R.~J.~DiPerna, A.~J.~Majda, Oscillations and concentrations in weak solutions of the incompressible fluid equations. {\em Comm.~Math.~Phys.}~{\bf 108}~(1987), no.~4, 667--689.
	
	\bibitem{E77} D.~G.~Ebin, The motion of slightly compressible fluids viewed as a motion with strong constraining force. {\em Ann.~Math.~(2)}~{\bf 105}~(1977), no.~1, 141--200.
	
	\bibitem{F14} E.~Feireisl, Maximal dissipation and well-posedness for the compressible Euler system. {\em J.~Math.~Fluid~Mech.}~{\bf 16}~(2014), no.~3, 447--461.
	
	\bibitem{FKM19} E.~Feireisl, C.~Klingenberg, S.~Markfelder, On the low Mach number limit for the compressible Euler system. {\em SIAM~J.~Math.~Anal.}~{\bf 51}~(2019), no.~2, 1496--1513.
	
	\bibitem{FL18} E.~Feireisl, M.~Luk\'a\v{c}ov\'a-Medvid'ov\'a, Convergence of a mixed finite element-finite volume scheme for the isentropic Navier-Stokes system via dissipative measure-valued solutions. {\em Found.~Comput.~Math.}~{\bf 18}~(2018), no.~3, 703--730.
	
	\bibitem{FM99} I.~Fonseca, S.~Müller, $\mathcal{A}$-quasiconvexity, lower semicontinuity, and Young measures. {\em SIAM~J.~Math.~Anal.}~{\bf 30}~(1999), no.~6, 1355--1390.
	
	\bibitem{G21} D.~Gallenmüller, Müller-Zhang truncation for general linear constraints with first or second order potential. {\em Calc.~Var.~Partial~Differential~Equations}~{\bf 60}~(2021), no.~3, 25~pp.
	
	\bibitem{GW20} D.~Gallenmüller, E.~Wiedemann, On the selection of measure-valued solutions for the isentropic Euler system. {\em J.~Differential~Equations}~{\bf 271}~(2021), no.~1, 979--1006.
	
	\bibitem{GW21} D.~Gallenmüller, E.~Wiedemann, Which measure-valued solutions of the monoatomic gas equations are generated by weak solutions? {\em arXiv preprint}~\href{https://arxiv.org/abs/2109.09513}{arXiv:2109.09513}~(2021).
	
	\bibitem{GSW15} P.~Gwiazda, A.~\'Swierczewska-Gwiazda, E.~Wiedemann, Weak-strong uniqueness for measure-valued solutions of some compressible fluids. {\em Nonlinearity}~{\bf 28}~(2015), no.~11, 3873--3890.
	
	\bibitem{KP91} D.~Kinderlehrer, P.~Pedregal, Characterizations of young measures generated by gradients. {\em Arch.~Ration.~Mech.~Anal.}~{\bf 115}~(1991), no.~4, 329--365.
	
	\bibitem{KP94} D.~Kinderlehrer, P.~Pedregal, Gradient Young measures generated by sequences in Sobolev spaces. {\em J.~Geom.~Anal.}~{\bf 4}~(1994), no.~1, 59--90.
	
	\bibitem{KM81} S.~Klainerman, A.~Majda, Singular limits of quasilinear hyperbolic systems with large parameters and the incompressible limit of compressible fluids. {\em Comm.~Pure~Appl.~Math.}~{\bf 34}~(1981), no.~4, 481--524.
	
	\bibitem{M21} S.~Markfelder, Convex integration applied to the multi-dimensional compressible Euler equations. Lecture Notes in Mathematics, Springer, Cham, 2021.
	
	\bibitem{M99} S.~Müller, A sharp version of Zhang's theorem on truncating sequences of gradients. {\em Trans.~Amer.~Math.~Soc.}~{\bf 351}~(1999), no.~11, 4585--4597.
	
	\bibitem{N93} J.~Neustupa, Measure-valued solutions of the Euler and Navier-Stokes equations for compressible barotropic fluids. {\em Math.~Nachr.}~{\bf 163}~(1993), no.~1, 217--227.
	
	\bibitem{R18} B.~Rai\c{t}\u{a}, Potentials for $\mathcal{A}$-quasiconvexity. {\em Calc.~Var.~Partial~Differential~Equations}~{\bf 58}~(2019), no.~3, Art.~105.
	
	\bibitem{Rindler} F.~Rindler, Calculus of Variations. Unversitext, Springer, Cham, 2018.
	
	\bibitem{SW12} L.~Sz\'ekelyhidi~Jr., E.~Wiedemann, Young measures generated by ideal incompressible fluid flows. {\em Arch.~Rational.~Mech.~Anal.}~{\bf 206}~(2012), no.~1, 333--366.
	
	\bibitem{W11} E.~Wiedemann, Existence of weak solutions for the incompressible Euler equations. {\em Ann.~Inst.~H.~Poincar\'e~Anal.~Non~Lin\'eaire}~{\bf 28}~(2011), no.~5, 727--730.
	
	\bibitem{W18} E.~Wiedemann, Weak-strong uniqueness in fluid dynamics. {\em London~Math.~Soc.~Lecture~Note~Ser.}~{\bf 452}~Cambridge Univ. Press, Cambridge (2018), 289--326.
\end{thebibliography}
\end{document}